\documentclass[10pt]{article}
\usepackage{flafter,amsmath,amssymb,latexsym,psfrag,graphicx,color,indentfirst}
\usepackage[margin=2.5cm]{geometry}
\usepackage{multirow}
\usepackage{subfig}
\usepackage{algorithm}
\usepackage{algorithmic}

\usepackage[colorlinks,
linktocpage=true,
linkcolor=black,
citecolor=black,
urlcolor=black,
anchorcolor=black
]{hyperref}
\usepackage{amsmath}
\usepackage{amsthm}
\usepackage{cases}
\usepackage{color}
\usepackage{ulem}
\usepackage{comment}
\usepackage{csquotes}
\usepackage{diagbox}

\usepackage[numbers,sort&compress]{natbib}

\DeclareGraphicsExtensions{.eps,.pdf,.jpg,.png}
\usepackage{epstopdf}
\allowdisplaybreaks

\newtheorem{theorem}{Theorem}[section]
\newtheorem{definition}[theorem]{Definition}
\newtheorem{lemma}[theorem]{Lemma}
\newtheorem{remark}[theorem]{Remark}

\newtheorem{corollary}[theorem]{Corollary}

\newtheorem{example}[theorem]{Example}
\newcommand{\red}[1]{\textcolor{red}{#1}}

\numberwithin{equation}{section}
\numberwithin{figure}{section}

\begin{document}
\setlength\arraycolsep{2pt}
\date{\today}
       
\title{Approximate peak time to time-domain fluorescence diffuse optical tomography for nonzero fluorescence lifetime}
\author{Shuli Chen,\; Junyong Eom\thanks{Corresponding author},\; Gen Nakamura,\; Goro  Nishimura}

\maketitle
\begin{abstract}
This paper concerns an inverse problem for fluorescence diffuse optical tomography (FDOT) reconstructing locations of multiple point targets from the measured temporal response functions. The targets are multiple fluorescent point objects with a nonzero fluorescence lifetime at unknown locations. Peak time,  when the temporal response function of the fluorescence reaches its maximum, is a robust parameter of the temporal response function in FDOT because it is most less suffered by the artifacts, such as noise, and is easily determined by experiments. 
%The FDOT process is modeled by two diffusion equations coupled with the source term \red{given by} a convolution with the exponential decay \red{term} of the fluorescence emission \red{delayed \sout{described}} by the lifetime. \red{We focus on the peak time of time-domain data, expecting a unique existence intuitively. \sout{Physically, a unique peak time can be clearly observed in the time-domain data.} 
%To use the peak time in FDOT, we derive the explicit expresssion for approximation peak time formula.
We derive an approximate peak time equation based on asymptotic analysis in an explicit way in the case of nonzero fluorescence lifetime when there are single and multiple point targets. The performance of the approximation is numerically verified.
Then, we develop a bisection algorithm to reconstruct the location of a single point target from the algorithm proposed in
\cite{Chen2023} for the case of zero fluorescence lifetime.
Moreover, we propose a boundary-scan algorithm for the reconstruction of locations of multiple point targets.
Finally, several numerical experiments are implemented to show the efficiency and robustness of the addressed algorithms.
\end{abstract}



\vspace{10pt}
\noindent Addresses:

\noindent S. Chen:
School of Mathematics, Southeast University, Nanjing 210096, P. R. China. \\
\noindent 
E-mail: {\tt sli\_chen@126.com} \;
ORCID ID: {\tt 0009-0000-6563-5312}\\

\noindent J. Eom:  
Research Institute for Electronic Science, Hokkaido University, Sapporo 001-0020, Japan.\\
\noindent 
E-mail: {\tt eom@es.hokudai.ac.jp} \;
ORCID ID: {\tt 0000-0002-2749-3322}\\

\noindent G. Nakamura:  Department of Mathematics, Hokkaido University, Sapporo 060-0810,
Research Institute for Electronic Science, Hokkaido University, Sapporo 001-0020, Japan.\\
\noindent 
E-mail: {\tt nakamuragenn@gmail.com} \; 
ORCID ID: {\tt 0000-0002-7911-8612}\\

\noindent G. Nishimura:  
Research Institute for Electronic Science, Hokkaido University, Sapporo 001-0020, Japan.\\
\noindent 
E-mail: {\tt gnishi@es.hokudai.ac.jp} \;
ORCID ID: {\tt 0000-0003-4330-2626}\\

\noindent {\bf Acknowledgement}:
The first author was supported by the National Natural Science Foundation of China (No. 12071072) and China Scholarship Council (No. 202106090240). The second author was supported by the JSPS KAKENHI (Grant Number 23K19002).
The third author was supported by the JSPS KAKENHI (Grant Number JP22K03366).
The last author was supported by the JSPS/MEXT KAKENHI (Grant Numbers JP19K04421 and JP23H04127).

\newpage
%%%% Start %%%%%%

\section[Background]{Introduction}
Fluorescence diffuse optical tomography (FDOT) is one of the imaging techniques using fluorescence light from fluorophores in highly scattering media. In particular, this technique is essential in biological or medical applications for tissues \textit{in vivo} to visualize specific diseases and biological activities using fluorescent probes by reconstructing their unknown regions from the measurements outside tissue \cite{Ammari2020, Mycek2004, Vasilis2002}. Although the importance of imaging inside tissue is well understood, the problems caused by the highly scattering media do not allow using standard imaging techniques, requiring a special imaging technique, FDOT. This imaging technique is also called fluorescence molecular tomography (FMT) and is categorized as a kind of diffuse optical tomography (DOT) using fluorescence. In highly scattering media like biological tissue, the light path is not straight anymore, and the repeating scattering makes the light propagation an energy dissipation process approximately described by a diffusion equation. %\sout{In this medium, the light propagation is considered as an energy dissipation by random scattering.} 
 Eventually, %\red{\sout{the spatial distribution of fluorescence becomes significantly blurred, and the quantitative information is lost. Therefore, instead of the common simple imaging techniques,}} 
 a reconstruction method based on the light propagation model is required to visualize the three-dimensional fluorescence distribution, recovering the quantitative information by measurements on the boundary of the medium \cite{Jiang2011, Jiang2022}. 
%\red{\sout{In FDOT, two light propagation processes for excitation and fluorescence emission lights are coupled. Excitation light injected at the boundary of the medium propagates through the medium and reaches the location of the fluorophores. Some of the excitation light is absorbed by the fluorophores, exciting the fluorophores, then after a moment, the fluorophores emit the fluorescence light at a longer wavelength than that of the excitation light. The fluorescence light again propagates through the medium until the fluorescence light reaches a detector located at the boundary of the medium}} %\cite{Jiang2011, Jiang2022}. 
The FDOT is categorized by the measurement data types \cite{Hebden1997, Hoshi2016},  
%The mathematical model of FDOT can be categorized by measuring different physical quantities \cite{Arridge2009,Hoshi2016}, 
such as the steady-state fluorescence intensity (CW method)\cite{Herve2011, Ducros2011}, the temporal response function of the fluorescence intensity (time-domain method)\cite{Lam2005, Nishimura2017}, the phase and demodulation of the fluorescence intensity (frequency-domain method)\cite{Ammari2014, Milstein2003} and a hybrid method of them \cite{Papadimitriou2020}. We choose the time-domain method because the temporal response function has direct information on the distribution of optical paths determined by the geometry, the position of the injection point of the excitation light, the distribution of the fluorophores, and the detection point of the fluorescence. %\sout{These are usually combined with the geometrical information of the excitation light source and the detector for the input of the FDOT. These data types are determined by the devices in the experiment and have advantages and disadvantages in realistic conditions. }}

In this paper, we analyze the temporal response function. The measured temporal response function is a set of the intensities of detected light in a certain time period at certain times after the instantaneous light injection like a delta function, and the time reflects the optical path length of a trajectory in the medium \cite{Hebden1991}, and it delays, in our case, by the staying time at the excited state of fluorophore determined by the fluorescence lifetime. Namely, the measured temporal response function is a discretized temporal response function, which is given by the solution of the light propagation model and reflects the optical path length distribution function associated with the target location. Then, we focus on the peak time, which will be defined as the time when the function becomes maximum, to characterize the temporal response function for the inputs of the reconstruction of targets. The maximum position of the temporal response function is least affected by the noises and artifacts due to environmental contamination and should be robustly determined. In addition, the peak position reflects the most likely optical path length for the light in the variety of travels from the source to the detector, which is determined by the distances between the source, the target, and the detector.

%The temporal response of the fluorescence intensity (time-domain data) records the fluorescence intensity over time when the excitation light is injected. 
%Compared with the measured data of other physical quantities, the unique advantage of this data is that it reflects the optical path length distribution \cite{Arridge92}. 
%More precisely, if a single target (i.e., fluorophore) produces fluorescence, the temporal response function is related to the distance to that target. This relation will guide in reconstructing the location of the target. Therefore, this paper will consider this data for our inverse problem.

Next, we start to formulate our inverse problem. We first consider the measurements of the fluorescence targets in tissue like the human chest, which is considerably larger than the measurable distances, resulting in the boundary being an approximately infinite plane. Thus, we consider a medium $\Omega:={\mathbb R}_+^3$ with boundary $\partial\Omega$. Let $u_e$ and $U_m$ be excitation and emission light for the fluorescence process, respectively. These two processes can be modeled as the following coupled diffusion equations \cite{Liu2022}:
\begin{equation}\label{ue_sys}
\left\{
\begin{array}{ll}
\left(v^{-1} \partial_t-D \Delta+\mu_a\right) u_e=0, & (x, t) \in \Omega \times(0, \infty), \\ 
u_e=0, & (x, t) \in \bar{\Omega} \times\{0\}, \\
\partial_\nu u_e+\beta u_e=\delta\left(x-x_s\right) \delta(t), & (x, t) \in \partial \Omega \times(0, \infty)
\end{array}
\right.
\end{equation}
and
\begin{equation}\label{um_sys}
\left\{
\begin{array}{ll}
\left(v^{-1} \partial_t-D \Delta+\mu_a\right) U_m=
\mu (f_\ell \ast u_e), & (x, t) \in \Omega \times(0, \infty), \\ 
U_m=0, & (x, t) \in \bar{\Omega} \times\{0\}, \\
\partial_\nu U_m+\beta U_m=0, & (x, t) \in \partial \Omega \times(0, \infty).
\end{array}
\right.
\end{equation}
Here, $\partial_\nu:=\nu \cdot \nabla$ is the exterior normal derivative, $v$ is the speed of light in the medium, $D$ is the diffusion constant, $\mu_a$ is the absorption coefficient, and  $\beta=b/D>0$ is a positive constant coming from the Fresnel reflection at the boundary, $b\in[0,1]$, due to the refractive index mismatch at the boundary. Also, $\delta(\cdot)$ denotes the delta function, and $x_s\in\partial\Omega $ is the position of the point source where the excitation light is injected. Further, the source term, which corresponds to the fluorescence emission from the fluorophores, for $U_m$ on the right-hand side of \eqref{um_sys} is given as  
\begin{equation}\label{source term}
\begin{split}
\mu (f_\ell \ast u_e) (x,t) 
&:= \mu(x) \int_0^t \ell^{-1}e^{-\frac{t-s}{\ell}} u_e(x, s; x_s)\,{\rm d}s,
\end{split}
\end{equation}
where  $\mu(x)>0$ is the absorption coefficient of a fluorophore, and $f_\ell\ast u_e$ is the convolution of $u_e$ and the fluorescence decay function, $f_\ell(t):=\ell^{-1}e^{-t/\ell},\,t\ge0$, with the fluorescence lifetime $\ell>0$. Note that if $\ell\rightarrow 0$, it is easy to see that $\mu(f_\ell\ast u_e)(x,t)=\mu(x)u_e(x,t;x_s)$. Hence, we extend the definition of \eqref{source term} to the case that the zero fluorescence lifetime $(\ell=0)$ by defining its right-hand side by $\mu(x)u_e(x,t;x_s)$.
\medskip

Then, our inverse problem is formulated as follows.

\noindent
\textbf{Inverse Problem}: Let 
\begin{equation}\label{mult_targets}
\mu(x)=\sum_{j=1}^J c_j\delta(x-x_c^{(j)}),
\end{equation}
where each $x_c^{(j)}$ is the location of $j$-th unknown target, and  each unknown $c_j$ is the absorption strength committed fluorescence by the target at $x_c^{(j)}$.
Then, reconstruct each $x_c^{(j)}$ from the measured peak times given as
$$t_{peak}^{(n)}:=t_{peak}(x_d^{(n)},\,x_s^{(n)})=\underset{t>0}{\arg \max}\,U_m(x_d^{(n)},\,t;\,x_s^{(n)}),\,n=1,\,2,\,\cdots,\,N,$$
where  $\left\{\{x_d^{(n)},\,x_s^{(n)}\}\right\}_{n=1}^N$ are $N$ sets of S-D pairs consisting of detector points $\{x_d^{(n)}\}_{n=1}^N$ and source points $\{x_s^{(n)}\}_{n=1}^N$ located on $\partial\Omega$. For $J=1$ and $J\geq 2$, we call $\mu$ single point target and multiple point targets, respectively.

\medskip

%The studies of FDOT can be divided into two categories according to different research methods. One is a modeling problem and the other is a data processing problem. As mentioned above, the model of FDOT is related to the data type, which will lead to different studies.  On the other hand, FDOT is an ill-posed inverse problem discussed in enormous studies, and it is known that the initial guess and the regularization to confine the problem are ones of crucial points to solve FDOT \cite{Arridge2009}. 
%To improve the image quality, we  refer to \cite{Baritaux2010, Han2010Yang} for $L^p$ regularization, \cite{Abascal2011, Behrooz2012} for total variation regularization, \cite{Han2010, Ammari2020, Mohajerani2007} for sparsity regularization and some references therein.

Now, we briefly review some related works to the mentioned inverse problem in historical order, where the peak time of the time-domain data is studied or further used to solve the inverse problem. In \cite{Hall2004,Hall2010}, the authors considered the case of $\ell>0$ and $\beta=0$. They reconstructed the depth of the point target by numerically calculating the peak time without giving any of the formulas derived in a mathematically rigorous way. Then in \cite{Eom2023a}, the authors considered the case of $\ell=0$. By using the asymptotic analysis for the formula of the solution to \eqref{um_sys}, they derived explicit expressions of the approximate peak time equations for the cases $\beta=0$, $\beta>0$, and $\beta=\infty$ in \eqref{ue_sys} and \eqref{um_sys}.
Further, in \cite{Chen2023}, the authors of this paper gave a better expression for the asymptotic behavior of the mentioned solution for the case $\ell=0, \beta>0$ and derived an approximate peak time equation. Also, this equation led them to propose a bisection reconstruction algorithm to reconstruct the location of the point target and verify its accuracy and robustness.

\begin{comment}
In this paper, we derive an approximate peak time equation and propose a bisection reconstruction algorithm for the case of $\ell>0$ in a mathematically rigorous way. We first derive an asymptotic formula of the time integral of the zero-lifetime solution $u_m$ in Theorem \ref{solbehaviortheorem2}, where we used the asymptotic formula of $u_m$ given in \cite{Chen2023}. From \eqref{expression of tilde um}, $U_m$ is be represented by a convolution of the exponential decay of the fluorescence emission described by $\ell$. By integration by parts, we derive an asymptotic formula of the time-derivative of $U_m$ given in \eqref{t-deri_Um} as $\ell$ goes to infinity. Using the asymptotic formula given in Theorem \ref{solbehaviortheorem2} and omitting a non-vanishing function, we define the approximate peak time equation and its unique positive root as the approximate peak time. The properties of the approximate peak time related to the distance $|x_d - x_c|^2+ |x_s - x_c|^2$ are rigorously proven and shown in Theorem \ref{thm_orderpeak}, which are also numerically verified to the peak time.  We propose the bisection reconstruction algorithm to reconstruct the location of unknown point target with two stages. The first two coordinates of the location are reconstructed by using bisection method taking the measured peak time as an index. Then, its depth are reconstructed by solving the approximate peak time equation. 
\end{comment}

We provide an asymptotic behavior of the solution $U_m$ to \eqref{um_sys} and derive an approximate peak time equation in the case of nonzero fluorescence lifetime, $\ell >0$. Since the effect from the fluorescence decay function $f_{\ell}$ in \eqref{source term} is not negligible in the practical range of  $\ell >0$, resulting in a clear difference in the peak times between zero lifetime solution and nonzero fluorescence lifetime solution $U_m$, we consider the case $\ell\gg1$, i.e. large enough fluorescence lifetime. 
Then, the asymptotic behavior of the solution $U_m$   
is written in terms of the integral (See \eqref{Um})
$$
\int_0^t u_m(s) \; ds,
$$
where $u_m$ is the solution to \eqref{um_sys} with zero fluorescence lifetime $\ell =0$. We can approximate the above integral by replacing $u_m$ as $u_m^a$, where $u_m^a$ is the asymptotic profile of $u_m$ when the depth of the point target is large enough (See Lemma \ref{solbehaviortheorem1}). In Theorem \ref{solbehaviortheorem2}, we provide the asymptotic behavior for 
\begin{equation}
\label{intasymp}
\int_0^t u_m^a(s) \; ds \sim k^{-\frac{3}{4}} \left( \pi \lambda \right)^\frac{1}{2}
u^a_m(\lambda k^{-\frac{1}{2}}),\,\,\lambda\gg1 \quad {\rm with} \quad  
k := \mu_a v, \quad \lambda^2 := \frac{|x_d-x_c|^2+|x_s-x_c|^2}{2vD},
\end{equation}
which enables us to derive the asymptotic behavior for the solution $U_m$ to \eqref{um_sys} and define approximate peak time $t>0$ as a solution of
\begin{equation}\label{NonlinearScheme2}
\begin{split}
\lambda e^{-\frac{(\sqrt{k}t-\lambda)^2}{t}} = \pi^{\frac{1}{2}} \ell^{-1} t^{\frac{3}{2}}.
\end{split}
\end{equation}
Here, we note that $\lambda>0$ is an important parameter, which is a constant times the square root of the sum of two squares of distances which are the distances from the target $x_c$ to the detector $x_d$ and the source $x_s$.

The advantage of the proposed approximate peak time equation \eqref{NonlinearScheme2} is that it has an explicit and simple form, which is derived from the asymptotic analysis of the solution 
$U_m$ to \eqref{um_sys}. 
As far as we know, there is less study on the asymptotic of the peak time in a rigorous way in the case of nonzero fluorescence lifetime $\ell >0$. Then, it is verified numerically that the approximate peak time has an excellent accuracy to the peak time of the solution $U_m$ in \eqref{um_sys} under the practical range of optical parameters $\mu_a,\,D,\,\ell$, and the depth of target $x_{c_3}$ (see Figure \ref{fig_tps_diffpara}). Next, we apply the approximate peak time for FDOT. 
When there is a single point target, we can develop the bisection reconstruction algorithm for the case $\ell >0$, where the algorithm itself has been already given in \cite{Chen2023} for the case $\ell =0$ and see that the algorithm is less time-consuming, efficient, potentially robust, and accurate. 
Even the equation \eqref{NonlinearScheme2} is written in the case of a single point target, it can be generalized for the case of multiple point targets depending on the parameter $\lambda >0$ (See \eqref{tp_formulas}). 
Finally, a boundary-scan algorithm is proposed to reconstruct multiple point targets, and it is verified numerically. 

\begin{comment}
The major novelties of the paper are that (i) we give the approximate peak time equation with an explicit expression, which reflects the optical path length distribution and depends on the actual physical parameters. (ii) The approximate peak time possesses the same properties of the peak time and has a nice accuracy. (iii) The bisection reconstruction algorithm is less time-consuming, efficient, potentially robust and accurate to reconstruct the location of unknown point target. 
\end{comment}

The rest of this paper is organized as follows. In Section \ref{sec_asym}, we introduce the asymptotic behavior of zero fluorescence lifetime solution and obtain the asymptotic behavior of its time integration.  
In Section \ref{sec_peak}, we derive the approximate peak time equation and define the approximate peak time. The performance of the approximation is numerically verified. 
In Section \ref{sec_algo}, the mathematical properties between the peak time and the location of a target are rigorously studied. Based on this, we propose the bisection reconstruction algorithm and the boundary-scan reconstruction algorithm for the cases of single point and multiple point targets, respectively. In Section \ref{sec_exp}, several examples are tested to show the efficiency of the proposed reconstruction algorithms. %Then, we derive the localized approximate peak time equation and propose the localized reconstruction algorithm with a numerical verification for the case of the multiple point targets.  
Finally, we conclude in Section \ref{sec_conc}.

\section[A]{Asymptotic behavior of solution}
\label{sec_asym}
In this section, we consider the asymptotic behavior of the solution $U_m$ to \eqref{um_sys} for single point and multiple point targets, respectively.

\subsection{Single point target}

In this subsection, we focus on the single point target case, i.e., $J=1$ in \eqref{mult_targets}. For simplicity, we suppress the superscript $(1)$ of $x_c^{(1)}$ and subscript $1$ of $c_1$ in \eqref{mult_targets}.  
By \eqref{source term}, the solution $U_m$ to \eqref{um_sys} is given as
\begin{equation}\label{expression of tilde um}
U_m = {\mathcal K} \ast \left( \mu f_\ell \ast u_e \right) 
=   f_\ell \ast \left( {\mathcal K} \ast \left[\mu u_e\right] \right), 
\end{equation}
where ${\mathcal K}$ is the Green function associated with \eqref{um_sys}.
Here we denote $u_m :={\mathcal K} \ast [\mu u_e]$ by the zero lifetime solution, which is independent of fluorescence lifetime $\ell>0$. 
For any given $x_d,x_s\in\partial\Omega$,  the solution $U_m$ to \eqref{expression of tilde um} becomes 
\begin{equation}\label{Um}
\begin{split}
U_m(x_d,t; x_s) &= \int_0^t \ell^{-1}e^{-\frac{t-s}{\ell}} u_m (x_d,s; x_s)  \,{\rm d}s \\
&= \ell^{-1} \int_0^t  u_m (x_d,s; x_s)  \,{\rm d}s - \int_0^t  \ell^{-2}e^{-\frac{t-s}{\ell}} \left[ \int_0^s u_m (x_d,\delta; x_s) \,{\rm d}\delta \right] \,{\rm d}s \\
&= \ell^{-1} \int_0^t  u_m (x_d,s; x_s)  \,{\rm d}s - \ell^{-2} \int_0^t  \left[ \int_0^s u_m (x_d,\delta; x_s) \,{\rm d}\delta \right] \,{\rm d}s + O \left( \ell^{-3} \right)
\end{split}
\end{equation}
as $\ell \gg 1$. We define the approximate solution of $U_m$ by
\begin{equation}\label{App_Um}
\begin{split}
U_m^a(t) &:= \ell^{-1} \int_0^t  u_m (s)  \,{\rm d}s - \ell^{-2} \int_0^t  \int_0^s u_m (\delta) \,{\rm d}\delta \,{\rm d}s \\
&=  \ell^{-1} \int_0^t  u_m (s)  \,{\rm d}s - 
\ell^{-2} \int_0^t (t-s) u_m (s)   \,{\rm d}s.
\end{split}
\end{equation}
%\medskip

Now, our aim is to study the asymptotic behaviors of $u_m$ and its time integration. 
We first give the expression for $u_m$.
The Green function ${\mathcal K}$ is represented by
\begin{equation}\label{expression of K}
{\mathcal K}(x, y; t) = \frac{ve^{-v\mu_a t}}{(4\pi vD t)^{\frac{3}{2}}}
e^{-\frac{(x_1-y_1)^2+(x_2-y_2)^2}{4vDt}} {\mathcal K}_3(x_3,y_3;t),
\end{equation}
where ${\mathcal K}_3$ satisfies the Robin boundary condition, which is 
\begin{equation*}
\left\{
\begin{array}{ll}
{\mathcal K}_3(x_3,y_3;t)= e^{-\frac{{(x_3+y_3)}^2}{4vDt}} +e^{-\frac{{(x_3-y_3)}^2}{4vDt}} -
2\beta\sqrt{\pi vDt}e^{\beta (x_3+y_3)+\beta^2vDt}
\mathop{\mathrm{erfc}}
\left(\frac{x_3+y_3+2\beta vDt}{\sqrt{4vDt}}\right),\vspace{5pt} \\
{\rm erfc}(\eta)=\frac{2}{\sqrt{\pi}}\int_\eta^\infty e^{-s^2} \, {\rm d}s, \; \; \eta\in \mathbb{R},
\end{array}
\right.
\end{equation*}
%We have used here and will use in this section, for any three dimensional 
where $x_j$ denotes $j$-th component of the three dimensional vector $x$. 
For the zero lifetime solution $u_m$, since the solution of $u_e$ to \eqref{ue_sys} is given as  $u_e = D \times {\mathcal K}$, we have the following expression
\begin{equation}\label{um-point2}
\begin{split}
u_m(x_d,t; x_s) = \frac{ce^{-v\mu_a t}}{16\pi^3 D^2v} \int_0^t \big((t-s)s\big)^{-\frac{3}{2}}  &e^{-\frac{|x_d-x_c|^2}{4vD(t-s)}} e^{-\frac{|x_s-x_c|^2}{4vDs}} \\
& \times {\hat {\mathcal K}}_3(x_{c_3};t-s) {\hat {\mathcal K}}_3(x_{c_3};s) \,{\rm d}s,
\end{split}
\end{equation}
where
\begin{equation*}
{\hat {\mathcal K}}_3(x_{c_3};t) :=1-\beta\sqrt{\pi vDt} \,\exp{\left(\left(\frac{x_{c_3}+2\beta vDt}{\sqrt{4vDt}}\right)^2\right)}
\mathop{\mathrm{erfc}}
\left(\frac{x_{c_3}+2\beta vDt}{\sqrt{4vDt}}\right).
\end{equation*}
Here $|x|^2 = x_1^2 + x_2^2 +x_3^2$ for $x = (x_1,x_2,x_3) \in \mathbb{R}^3$.

\medskip
The next lemma describes the asymptotic behavior of ${u}_m$  for $x_{c_3}\gg1$ (see \cite[Theorem 2.2 and Remark 2.3]{Chen2023}).
\begin{lemma}\label{solbehaviortheorem1}
Let $x_d, \; x_s\in\partial\Omega$, and assume that
\begin{equation}\label{uma_constr}
\begin{split}
\Big| |x_d-x_c|^2-|x_s-x_c|^2 \Big| \le C t \quad {\rm for\; some} \quad C>0.
\end{split}
\end{equation}
Then, $u_m$ satisfies
\begin{equation}\label{1st_asym_um}
\begin{split}
u_m(t) = u^a_m(t) +
O \left( u^a_m(t) x_{c_3}^{{-1}}  \right),\;\;x_{c_3}\gg1,
\end{split}
\end{equation}
where
\begin{align}\label{defIt}
u^a_m(t):=& \frac{c \exp\left({-\mu_a vt }\right)}{8\pi^\frac{5}{2} v^\frac{1}{2} D^\frac{3}{2} } \left( \frac{1}{|x_d-x_c|} + \frac{1}{|x_s-x_c|} \right)  t^{-\frac{3}{2}} \nonumber\\
&\times
\exp\left( -\frac{|x_d-x_c|^2+|x_s-x_c|^2}{2vDt} \right)
\left(\frac{x_{c_3}}{x_{c_3}+\beta v D t}\right)^2.
\end{align}
\end{lemma}

\medskip
In the next lemma, we provide asymptotic behaviors of some integrals, which will be used to derive the asymptotic behavior of $\int_0^t u_m^a(s)\,{\rm d}s$ for $\lambda\gg1$.
\begin{lemma}\label{intlemma}
Let $\lambda >0$. Assume $f_\lambda \in C^1[\delta_1,\infty)$ for some $\delta_1>0$, and there exists $C>0$ such that 
\begin{equation}
\begin{split}
\label{fcondition}
| \partial_s f_\lambda(s) | \le C \lambda^{-1}
|  f_\lambda(s) | \quad {\rm for} \quad s \ge \delta_1.
\end{split}
\end{equation}
Then, for any $\delta_2 >\delta_1$,
\begin{equation}
\label{twointe}
\begin{split}
&{\rm i)}\; \int_{\delta_1}^{\delta_2} e^{-\lambda \eta} f_\lambda(\eta) \; {\rm d}\eta = 
\lambda^{-1}f_\lambda(\delta_1) e^{-\lambda \delta_1} + O\left( \lambda^{-3} f_\lambda(\delta_1) e^{-\lambda \delta_1}\right) \quad {\rm as} \quad \lambda \gg 1,\\
&{\rm ii)}\; \int_{\delta_1}^{\delta_2} e^{-\lambda \eta} (\eta-\delta_1)^{-\frac{1}{2}} f_\lambda(\eta) \; {\rm d}\eta = 
\pi^\frac{1}{2}\lambda^{-\frac{1}{2}}f_\lambda(\delta_1) e^{-\lambda \delta_1} + O\left( \lambda^{-\frac{5}{2}}  f_\lambda(\delta_1) e^{-\lambda \delta_1}\right) \quad {\rm as} \quad \lambda \gg 1.
\end{split}
\end{equation}
\end{lemma}
\begin{proof}
By \eqref{fcondition}, we obtain for any $\delta_2>\delta_1$
\begin{equation*}
\begin{split}
\int_{\delta_1}^{\delta_2} e^{-\lambda \eta} f_\lambda(\eta) \; {\rm d}\eta 
&= \lambda^{-1}f_\lambda(\delta_1) e^{-\lambda \delta_1} -  \lambda^{-1}
e^{-\lambda \delta_2} f_\lambda(\delta_2) + \lambda^{-1} 
 \int_{\delta_1}^{\delta_2} e^{-\lambda \eta} \partial_\eta f_\lambda(\eta) \; {\rm d}\eta \\
&= \lambda^{-1}f_\lambda(\delta_1) e^{-\lambda \delta_1} +
O\left( \lambda^{-3} f_\lambda(\delta_1) e^{-\lambda \delta_1}\right)
\end{split}
\end{equation*}
as $\lambda \gg 1$, which implies ${\rm i)}$ in \eqref{twointe}.
On the other hand, we obtain
\begin{equation}
\label{Ionetwo}
\begin{split}
\int_{\delta_1}^{\delta_2} e^{-\lambda \eta} (\eta-\delta_1)^{-\frac{1}{2}} f_\lambda(\eta) \; {\rm d}\eta 
&= e^{-\lambda \delta_1} \int_0^{\delta_2-\delta_1} e^{-\lambda \eta} \eta^{-\frac{1}{2}} f_\lambda(\eta+\delta_1) \; {\rm d}\eta \\
&=e^{-\lambda \delta_1} \left[ \int_0^{\delta_2-\delta_1} e^{-\lambda \eta} \eta^{-\frac{1}{2}} f_\lambda(\delta_1) \; {\rm d}\eta + \int_0^{\delta_2-\delta_1} e^{-\lambda \eta} \eta^\frac{1}{2} \partial_\eta f_\lambda(\tilde{\eta}_{\delta_1}) \; {\rm d}\eta \right] \\
&:=  I_1 + I_2 
\end{split}
\end{equation}
for some $\tilde{\eta}_{\delta_1}>\delta_1$.
Since
\begin{equation*}
\begin{split}
I_1 & =  f_\lambda(\delta_1)  e^{-\lambda \delta_1} \int_0^{\delta_2-\delta_1} e^{-\lambda \eta} \eta^{-\frac{1}{2}}  \; {\rm d}\eta=  f_\lambda(\delta_1)  e^{-\lambda \delta_1} \left[ \int_0^\infty - \int_{\delta_2-\delta_1}^\infty \right] e^{-\lambda \eta} \eta^{-\frac{1}{2}} \;{\rm d}\eta \\
& = 2 \lambda^{-\frac{1}{2}}f_\lambda(\delta_1) e^{-\lambda \delta_1} \int_0^\infty 
e^{-\eta^2} \; {\rm d}\eta + O\left(\lambda^{-1}f_\lambda(\delta_1) e^{-\lambda \delta_2} \right)  \quad \mbox{as} \quad \lambda \gg 1,
\end{split}
\end{equation*}
and by \eqref{fcondition}, we obtain
\begin{equation*}
\begin{split}
| I_2 | & \le C \lambda^{-1}|f_\lambda(\delta_1)| e^{-\lambda \delta_1}  \int_0^{\delta_2-\delta_1} e^{-\lambda \eta} \eta^{\frac{1}{2}}  \; {\rm d}\eta \\
&=  C |f_\lambda(\delta_1)|\lambda^{-1} e^{-\lambda \delta_1} \left[ -\lambda^{-1}
e^{-\lambda (\delta_2-\delta_1)} (\delta_2-\delta_1)^\frac{1}{2}
+ \frac{\lambda^{-1}}{2}  \int_0^{\delta_2-\delta_1} e^{-\lambda \eta} \eta^{-\frac{1}{2}}  \; {\rm d}\eta  \right] \\
&=O\left(\lambda^{-\frac{5}{2}} f_\lambda(\delta_1) e^{-\lambda \delta_1} \right) \quad \mbox{as} \quad \lambda \gg 1.
\end{split}
\end{equation*}
This together with \eqref{Ionetwo} implies
$$
\int_{\delta_1}^{\delta_2} e^{-\lambda \eta} (\eta-\delta_1)^{-\frac{1}{2}} f_\lambda(\eta) \; {\rm d}\eta = 
\pi^\frac{1}{2}\lambda^{-\frac{1}{2}}f_\lambda(\delta_1) e^{-\lambda \delta_1} + O\left( \lambda^{-\frac{5}{2}} f_\lambda(\delta_1)e^{-\lambda \delta_1}\right)
$$
as $\lambda \gg 1$, and the proof of Lemma \ref{intlemma} is complete.
\end{proof}
We are ready to derive the asymptotic behavior of the time integration for $u^a_m$ in \eqref{defIt}, which is closely related to the behavior of the solution $U_m$ to \eqref{um_sys}.
In the next Theorem \ref{solbehaviortheorem2}, we obtain the asymptotic of
$$
\int_0^t u^a_m(s) \; {\rm d}s  = \int_0^t e^{-ks - \frac{\lambda^2}{s}} f(s) \; {\rm d}s 
\quad \mbox{as} \quad \lambda \gg 1,
$$
where
\begin{equation}
\begin{split}
\label{defk}
&k := \mu_a v, \quad \lambda^2 := \frac{|x_d-x_c|^2+|x_s-x_c|^2}{2vD}, \\
&f(s) := \frac{c}{8\pi^\frac{5}{2} v^\frac{1}{2} D^\frac{3}{2} } \left( \frac{1}{|x_d-x_c|} + \frac{1}{|x_s-x_c|} \right)  s^{-\frac{3}{2}}\left(\frac{x_{c_3}}{x_{c_3}+\beta v D s}\right)^2.
\end{split}
\end{equation}
\begin{theorem}\label{solbehaviortheorem2}
Assume $t>\lambda k^{-\frac{1}{2}}$.
Then $u^a_m$ of \eqref{defIt} satisfies
\begin{equation}
\label{integralbehavior}
\int_0^t u^a_m(s) \; {\rm d}s = k^{-\frac{3}{4}} \left( \pi \lambda \right)^\frac{1}{2}
u^a_m(\lambda k^{-\frac{1}{2}}) + O \left( \lambda^{-\frac{3}{2}} u^a_m(\lambda k^{-\frac{1}{2}}) \right) 
\end{equation}
as $\lambda \gg 1$, where $k>0$ and $\lambda$ are as in \eqref{defk}.
\end{theorem}
\begin{proof}
Since $t>\lambda k^{-\frac{1}{2}}$, we separate the integral into two parts
\begin{equation}
\label{intseparation}
\begin{split}
\int_0^t u^a_m(s) \; {\rm d}s  &= \int_0^t e^{-ks - \frac{\lambda^2}{s}} f(s) \; {\rm d}s
= \int_0^{\frac{t}{\lambda}}  e^{-\lambda\left(k\zeta + \zeta^{-1}\right)} \lambda f(\lambda \zeta) \; {\rm d}\zeta \\
&= \lambda \left[\int_0^{k^{-\frac{1}{2}}} + \int_{k^{-\frac{1}{2}}}^{\frac{t}{\lambda}} \right]
 e^{-\lambda\left(k\zeta + \zeta^{-1}\right)}  f(\lambda \zeta) \; {\rm d}\zeta. 
\end{split}
\end{equation}
Set a new variable $\eta := k\zeta + \zeta^{-1}$ having 
$$
\zeta_{\pm} := \frac{\eta \pm \sqrt{\eta^2-4k}}{2k} \quad \mbox{with} \quad
{\rm d} \zeta_{\pm} = \frac{1}{2k} \left[ 1 \pm \eta (\eta^2-4k)^{-\frac{1}{2}} \right] {\rm d} \eta.
$$
We obtain
\begin{equation}
\begin{split}
\int_0^t u^a_m(s) \; {\rm d}s &= - \lambda\int_{2\sqrt{k}}^\infty e^{-\lambda \eta}  f(\lambda \zeta_{-}) \; {\rm d} \zeta_{-} +  \lambda \int_{2\sqrt{k}}^{\eta(\frac{t}{\lambda})} e^{-\lambda \eta}  f(\lambda \zeta_{+}) \; {\rm d} \zeta_{+} \\
&= - \frac{\lambda}{2k} \int_{2\sqrt{k}}^\infty e^{-\lambda \eta}  f(\lambda \zeta_{-}) \; {\rm d} \eta 
+ \frac{\lambda}{2k} \int_{2\sqrt{k}}^\infty e^{-\lambda \eta}  f(\lambda \zeta_{-}) 
\eta (\eta^2-4k)^{-\frac{1}{2}}\; {\rm d} \eta \\
&\quad + \frac{\lambda}{2k} \int_{2\sqrt{k}}^{\eta(\frac{t}{\lambda})} e^{-\lambda \eta}  f(\lambda \zeta_{+}) \; {\rm d} \eta
+ \frac{\lambda}{2k} \int_{2\sqrt{k}}^{\eta (\frac{t}{\lambda})} e^{-\lambda \eta}  f(\lambda \zeta_{+}) 
\eta (\eta^2-4k)^{-\frac{1}{2}}\; {\rm d} \eta.
\end{split}
\end{equation}
We apply Lemma \ref{intlemma} with $f_\lambda(s) = f(\lambda s)$ for $s\ge \delta_1 = 2\sqrt{k}$ to obtain  
\begin{equation}
\begin{split}
\int_0^t u^a_m(s) \; {\rm d}s &= - \frac{1}{2k}
f(\lambda k^{-\frac{1}{2}} ) e^{-2 \lambda \sqrt{k} }
+ \frac{\pi^\frac{1}{2}\lambda^\frac{1}{2}}{2k}
f(\lambda k^{-\frac{1}{2}} ) 2\sqrt{k} (4\sqrt{k})^{-\frac{1}{2}}
e^{-2 \lambda \sqrt{k} } + \frac{1}{2k}
f(\lambda k^{-\frac{1}{2}} ) e^{-2 \lambda \sqrt{k} } \\
&\qquad + \frac{\pi^\frac{1}{2}\lambda^\frac{1}{2}}{2k} 
f( \lambda k^{-\frac{1}{2}}) 2\sqrt{k} (4\sqrt{k})^{-\frac{1}{2}}
e^{-2 \lambda \sqrt{k} }
 + O\left( \lambda^{-\frac{3}{2}} f( \lambda k^{-\frac{1}{2}}) e^{-2 \lambda \sqrt{k} }\right) \\
 &=  k^{-\frac{3}{4}} \left( \pi \lambda \right)^\frac{1}{2} f( \lambda k^{-\frac{1}{2}})
 e^{-2 \lambda \sqrt{k} }  + O\left( \lambda^{-\frac{3}{2}} f( \lambda k^{-\frac{1}{2}}) e^{-2 \lambda \sqrt{k} }\right) \\
\end{split}
\end{equation}
as $\lambda \gg 1$. We remark $\zeta_\pm(2\sqrt{k}) = k^{-\frac{1}{2}}$, and by \eqref{defk}, $f_\lambda$ satisfies the condition \eqref{fcondition}.   
Since $u_m^a(\lambda k^{-\frac{1}{2}}) =  f( \lambda k^{-\frac{1}{2}})
 e^{-2 \lambda \sqrt{k} }$, we obtain  \eqref{integralbehavior}, and the proof of Theorem \ref{solbehaviortheorem2} is complete.
\end{proof}

\subsection{Multiple point targets}
In this subsection, we turn to consider the asymptotic behavior of $U_m$ defined by \eqref{expression of tilde um} for the multiple point targets $x_c^{(1)},\,x_c^{(2)},\,\cdots,\,x_c^{(J)}\in\Omega$. By using the principle of superposition, the zero lifetime solution, still denoted by $u_m$, can be defined by 
\begin{equation}
    u_m(x_d,\,t;\,x_s)=\sum_{j=1}^J u_m^{(j)}(x_d,\,t;\,x_s),
\end{equation}
where $x_d,\,x_s\in\partial\Omega$, each $u_m^{(j)},\,j=1,\,2,\,\cdots,\,J,$ is defined by \eqref{um-point2} replacing $x_c$ with $x_c^{(j)}$.
\medskip

The following Corollary gives the asymptotic behaviors of $u_m$ and its time integration. 
\begin{corollary}\label{cor_asym}
Let $x_d,\,x_s\in\partial\Omega$ and $x_c^{(1)},\,x_c^{(2)},\,\cdots,\,x_c^{(J)}\in\Omega$. Assume that there exists some $x_c^{(l)},\,l=1,\,2,\,\cdots,\,J,$ such that 
\begin{equation}\label{equ_dis_rela_mult}
|x_d-x_c^{(l)}|^2+|x_s-x_c^{(l)}|^2<|x_d-x_c^{(j)}|^2+|x_s-x_c^{(j)}|^2,\;\;1\leq j\neq l\leq J,
\end{equation}
and
$$\Big| |x_d-x_c^{(l)}|^2-|x_s-x_c^{(l)}|^2 \Big| \le C^{(l)} t \quad {\rm for\; some} \quad C^{(l)}>0.$$
Then, we have the asymptotic formula of $u_m$
\begin{align}\label{1st_asym_ums}
&u_m(t ;\,x_c^{(1)},\,x_c^{(2)},\,\cdots,\,x_c^{(J)})
=u_m^{a,(l)}(t) \big[1+ o(1)\big]
\end{align}
as $x_{c_3}^{(l)}\gg 1$, where $u_m^{a,(l)}$ can be defined by \eqref{defIt} replacing $x_c$ with $x_c^{(l)}$.

Moreover, assuming that $t>\lambda^{(l)} k^{-\frac{1}{2}}$, then
\begin{equation}
\label{integralbehavior_s}
\sum_{j=1}^J\int_0^t u^{a,(j)}_m(s) \; {\rm d}s = k^{-\frac{3}{4}} \left( \pi \lambda^{(l)} \right)^\frac{1}{2}
u^a_m(\lambda^{(l)} k^{-\frac{1}{2}})\big[1+ o(1)\big]
\end{equation}
as $\lambda^{(l)} \gg 1$, where $\lambda^{(l)}$ can be defined by \eqref{defk} replacing $x_c$ with $x_c^{(l)}$.
\end{corollary}
\begin{proof}
As generalizations of Lemma \ref{solbehaviortheorem1}, for each $x_c^{(j)},\,j=1,\,2,\,\cdots,\,J$ such that $\Big| |x_d-x_c^{(j)}|^2-|x_s-x_c^{(j)}|^2 \Big| \le C^{(j)} t$ for some $C^{(j)}>0$, then $u_m^{(j)}$ satisfies 
$$u_m^{(j)}(t) = u^{a,(j)}_m(t) +O \left( u^{a,(j)}_m(t) (x_{c_3}^{(j)})^{{-1}} \right),\;\;x_{c_3}^{(j)}\gg1.$$
It is obvious that the dominant part of each $u_m^{a,(j)}$ is the exponentially small term $\exp\left( -\frac{|x_d-x_c^{(l)}|^2+|x_s-x_c^{(l)}|^2}{2vDt} \right)$. Under the condition \eqref{equ_dis_rela_mult}, we obtain 
$$\sum_{j=1}^J u_m^{a,(j)}(t)=u_m^{a,(l)}(t)\big[1+ o(1)\big],$$
which implies \eqref{1st_asym_ums}. Based on this and by the same proof of Theorem \ref{solbehaviortheorem2}, we obtain \eqref{integralbehavior_s}.
\end{proof}

\begin{remark}\label{rem_mult_tp}
    The asymptotic behaviors \eqref{1st_asym_ums} and \eqref{integralbehavior_s} are related to $x_c^{(l)}$ due to \eqref{equ_dis_rela_mult}. In other words, we can derive the same asymptotic behaviors of $u_m$ and its time integration for each $x_c^{(l)},\,l=1,\,2,\,\cdots,\,J$,
    if there always exist $x_d,\,x_s\in\partial\Omega$ such that \eqref{equ_dis_rela_mult} for each $x_c^{(l)},\,l=1,\,2,\,\cdots,\,J$. 
\end{remark}

\begin{comment}
\begin{remark} When $t \le \lambda k^{-\frac{1}{2}}$, there is only one integration in \eqref{intseparation},
that is,
\begin{equation}
\begin{split}
\int_0^t u^a_m(s) \; {\rm d}s &= \lambda \int_0^{\frac{t}{\lambda}}
 e^{-\lambda\left(k\zeta + \zeta^{-1}\right)}  f(\lambda \zeta) \; {\rm d}\zeta 
 =- \lambda\int_{\frac{\eta t}{\lambda}}^\infty e^{-\lambda \eta}  f(\lambda \zeta_{-}) \; {\rm d} \zeta_{-}(\eta).
\end{split}
\end{equation}
By Lemma \ref{intlemma} with $\delta_1 = \frac{\eta t}{\lambda}$, we obtain
\begin{equation}
\label{oppositecase}
\begin{split}
\int_0^t u^a_m(s) \; {\rm d}s &=
\frac{\pi^\frac{1}{2}\lambda^\frac{1}{2}}{2k}
\left( \frac{kt}{\lambda} + \frac{\lambda}{t} \right)
\left( \frac{kt}{\lambda} + \frac{\lambda}{t} + 2\sqrt{k} \right)^{-\frac{1}{2}} u^a_m(t) + O\left( u^a_m(t) \right)
\\
&= \frac{\pi^\frac{1}{2}}{2k} \lambda t^{-\frac{1}{2}} u^a_m(t) + O\left( u^a_m(t) \right)
\end{split}
\end{equation}
as $\lambda \to \infty$.
\end{remark}
\end{comment}
\section{%Inversion scheme
Approximate peak time}
\label{sec_peak}
In this section, we derive approximate peak time equations and define approximate peak times for the cases of single point and multiple point targets, respectively. The accuracy of the approximate peak times is numerically verified.

\medskip

By \eqref{App_Um}, we look for an approximate peak time for $U_m$ as $t$, which satisfies 
\begin{equation}\label{t-deri_Um}
\begin{split}
\partial_t U_m^a(t) = \ell^{-1} u_m (t) - \ell^{-2} \int_0^t  u_m (s) \,{\rm d}s =0.
\end{split}
\end{equation}
In the following, approximating $u_m$ by \eqref{1st_asym_um} and \eqref{1st_asym_ums}, we consider the cases of single point and multiple point targets, respectively.

\begin{comment}
When $\lambda \ge tk^{\frac{1}{2}}$, by \eqref{oppositecase} we obtain 
\begin{equation}\label{NonlinearScheme2}
\begin{split}
\lambda = 2k \pi^{-\frac{1}{2}} \ell t^\frac{1}{2}.
\end{split}
\end{equation}
\end{comment}

\subsection{Approximate peak time for single point target}

For the single point target, replacing $u_m$ with $u_m^a$ in \eqref{t-deri_Um} leads to 
\begin{equation}\label{t-deri_Um1}
\begin{split}
\int_0^t  u_m^a (s) \,{\rm d}s = \ell u_m^a (t),
\end{split}
\end{equation}
where $u_m^a$ is as in \eqref{defIt}. 
By \eqref{defk} and Theorem \ref{solbehaviortheorem2}, \eqref{t-deri_Um1} becomes
$$
\ell e^{-\frac{kt^2+\lambda^2}{t}} f(t) = \ell u_m^a (t) \sim
k^{-\frac{3}{4}} \left( \pi \lambda \right)^{\frac{1}{2}} f( \lambda k^{-\frac{1}{2}})
 e^{-2 \lambda \sqrt{k} }, 
$$
and from
$$
\left(\frac{x_{c_3}+\beta vDt }{x_{c_3}+\beta vD \lambda k^{-\frac{1}{2}}}\right)^2 \to 1 \quad \mbox{as} \quad 
x_{c_3} \gg 1,
$$
we can approximate \eqref{t-deri_Um1} by
\begin{equation}\label{NonlinearScheme}
\begin{split}
\lambda e^{-\frac{(\sqrt{k}t-\lambda)^2}{t}} = \pi^{\frac{1}{2}} \ell^{-1} t^{\frac{3}{2}}.
\end{split}
\end{equation}
Based on \eqref{NonlinearScheme}, we first define the approximate peak time as the root of the approximate peak time equation 
\begin{equation}\label{peak time eq}
P(t;\lambda)=0,
\end{equation}
where $P(t;\lambda)$ is given as
\begin{equation}\label{func_peak}
P(t;\,\lambda):=\lambda e^{-\frac{(\sqrt{k}t-\lambda)^2}{t}} -\pi^{\frac{1}{2}} \ell^{-1} t^{\frac{3}{2}},\,t>\lambda k^{-\frac{1}{2}}.
\end{equation}
Then, we study the unique existence of the approximate peak time and numerically test its applicability to different physical situations. 

\begin{theorem}\label{thm_uniq_apppeak}
Let $x_c\in\Omega$ and $x_d,\,x_s\in\partial\Omega$, 
which are assumed to satisfy  
the constraint \eqref{uma_constr}, $x_{c_3}\gg 1$, and
$\ell>\pi^{\frac{1}{2}}k^{-\frac{3}{4}}\lambda^{\frac{1}{2}}$. Then,  there exists a unique root $t_{peak}^a$ of $P(t)$ such that
$$P(t)>0,\;\;\lambda k^{-\frac{1}{2}}<t<t_{peak}^a\;\;{\rm and}\;\; P(t)<0,\;\;t>t_{peak}^a.$$
%and write $P(t)=P(t;\,\lambda)$ to clarify that $P(t)$ depends on $\lambda=\lambda(x_d,\,x_s;\,x_c)$.
\end{theorem}

\begin{proof}
To show the unique existence of $t_{peak}^a$, let us first examine the monotonicity of $P(t;\,\lambda)$ for $t>\lambda k^{-\frac{1}{2}}$. 
Consider
\begin{equation*}
\partial_t P=\lambda e^{-\frac{(\sqrt{k}t-\lambda)^2}{t}}\times \left(-\frac{(kt^2-\lambda^2)}{t^2} \right) -\frac{3}{2}\pi^{\frac{1}{2}} \ell^{-1} t^{\frac{1}{2}}.
\end{equation*}
%It is reasonable to assume that the physical parameter $k=\mu_a v<4$ in a practical model. 
Then, we have 
$\partial_t P<0$ for $t>\lambda k^{-\frac{1}{2}}$, which means that $P(t;\,\lambda)$ is a monotonically decreasing function for $t>\lambda k^{-\frac{1}{2}}$. Hence, we only need to show that $P(t;\,\lambda)>0$ at $t=\lambda k^{-\frac{1}{2}}$ and $\underset{t\rightarrow\infty}{\lim}P(t;\lambda)<0$. The second one is obvious, and the first one follows from
$$P(\lambda k^{-\frac{1}{2}};\,\lambda)=\lambda\left( 1-\pi^{\frac{1}{2}}\ell^{-1}k^{-\frac{3}{4}}\lambda^{\frac{1}{2}}\right).$$
The proof of Theorem \ref{thm_uniq_apppeak} is complete.
\end{proof} 

Next, we numerically examine the performance of the approximate peak time for different physical parameters $D,\,\mu_a,\,\ell$, and the depth of the target, $x_{c_3}$. We evaluate the relative error of the approximate peak time as 
\begin{equation}\label{relerr_tp}
RelErr_t:=\frac{|t_{peak}-t_{peak}^a|}{t_{peak}},
\end{equation}
where the peak time $t_{peak}$ and the approximate peak time $t_{peak}^a$  are computed from \eqref{expression of tilde um} and \eqref{peak time eq}, respectively. Here, we compute $U_m(x_d,\,t;\,x_s)$ using a numerical integration with a time step $0.1\,{\rm ps}$, then find the peak time $t_{peak}$ from the discretized time point, which gives the maximum of $U_m(x_d,\,t;\,x_s)$. Both $t_{peak}$, $t_{peak}^a$ and $RelErr_t$ depend on the physical parameters $v,\,D,\,\mu_a,\,\ell,\,\beta$ and the depth $x_{c_3}$. 
The parameters $D$, $\mu_a$, and $v$ depend on the biological tissue types and conditions. Since there is a large variety of the reported values of $D$ and $\mu_a $\cite{Taroni2002, Bashkatov2011}, we choose some representative values for evaluating how the approximation performs under the practical ranges of these parameters. We assume that the refractive index is a fixed value of 1.37 because biological tissues dominantly consist of water, resulting in $\beta$ only depending on $D$ and the fixed light speed of 0.219~mm/ps. The fluorescence lifetime $\ell$ depends on the fluorophore molecule and its environment, but the value is usually less than a few nano-seconds for typical organic fluorophore molecules \cite{Lakowicz1999}. The target depth  $x_{c_3}$ is limited by a detection limit of about $30\,{\rm mm}$ \cite{Nishimura2024}.
If no otherwise specified, we always set $ x_{c_3}=20\,{\rm mm}$ and 
\begin{equation}\label{phys_para}
v=0.219\, {\rm mm/ps},\quad D=1/3\,{\rm mm},\quad \mu_a=0.1\,{\rm mm^{-1}},\quad \beta=0.5493\, {\rm mm^{-1}},\quad \ell=1000\,{\rm ps},
\end{equation}
which are typical values in biological tissues.
In Figure \ref{fig_tps_diffpara},
we show the numerical results for fixed S-D pair $\{x_{d},\,x_{s}\}=\{(14,\,10,\,0),\,(6,\,10,\,0)\}$, the projected location of the target $x_{c}=(10,\,10,\,x_{c_3})$ and changed $x_{c_3},\,D,\,\mu_a,\,\ell$.

Figures \ref{fig_tps_diffpara} (a) and (d)  indicate that both the peak time and approximate peak time are increasing, while the differences are decreasing with the fluorescence lifetime $\ell>0$. The decrease of the relative errors of the approximated peak time is consistent with our approximation based on large $\ell \gg 1$ as in \eqref{Um}. Figures \ref{fig_tps_diffpara} (b) and (e) show that both the peak time, approximate peak time, and the relative errors of the approximate peak time are decreasing with respect to absorption coefficient $\mu_a>0$. 
%\red{Further, the error becomes large for larger values of $D$ at a smaller $\mu_{a}$, while the difference of the error becomes invisible at a larger $\mu_{a}$. These observed behaviors are complex, and it is difficult to interpret them mathematically.} 
For fixed $\mu_a$, the relative error becomes smaller for smaller $D$, which is consistent with our approximation based on large $\lambda\gg 1$ as in \eqref{integralbehavior}. Figure \ref{fig_tps_diffpara} (c) shows that both the peak time and approximate peak time are increasing with respect to the depth of point target $x_{c_3}>0$ because the distance to the target becomes large. Interestingly, the increase is almost linear with the depth, and this observation is consistent with other studies \cite{Hall2004, Hall2010}. However, Figure \ref{fig_tps_diffpara} (f) indicates no decrease in the relative error in spite of using the approximation given in \eqref{1st_asym_um}. This is the counterpart to our previous result on the approximation peak time for the case of zero fluorescence lifetime $\ell =0$.

\begin{figure}[htp]
\centering
\begin{tabular}{lll}
(a) & (b) & (c) \\
\includegraphics[width=0.33\textwidth]{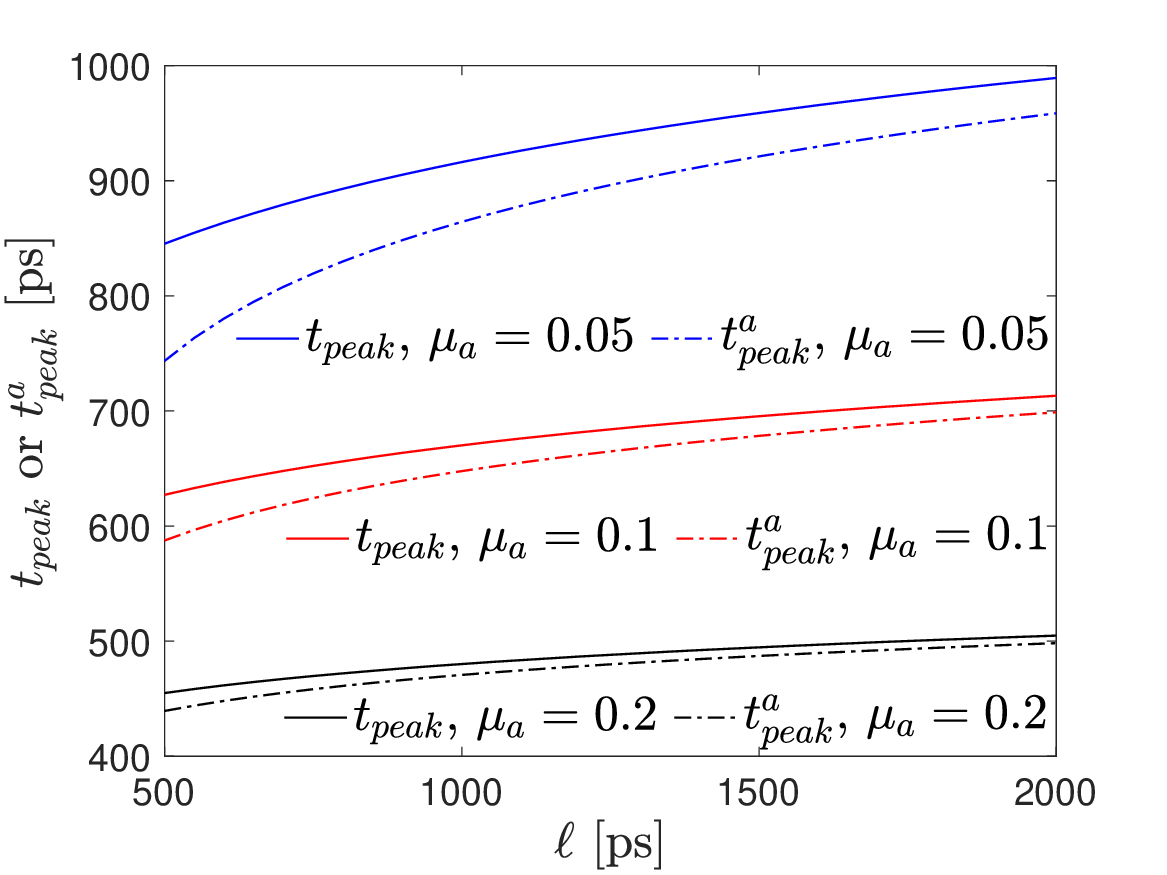}
& \includegraphics[width=0.33\textwidth]{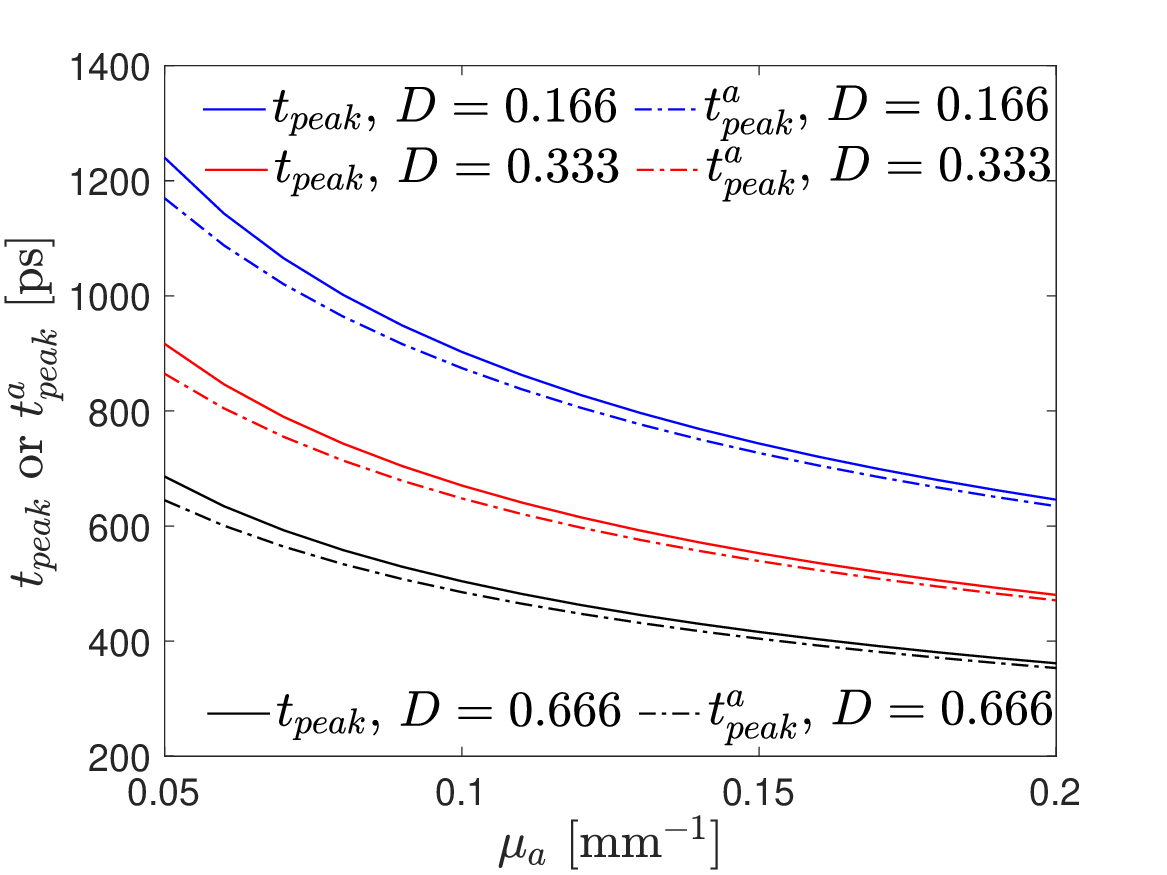}
& \includegraphics[width=0.33\textwidth]{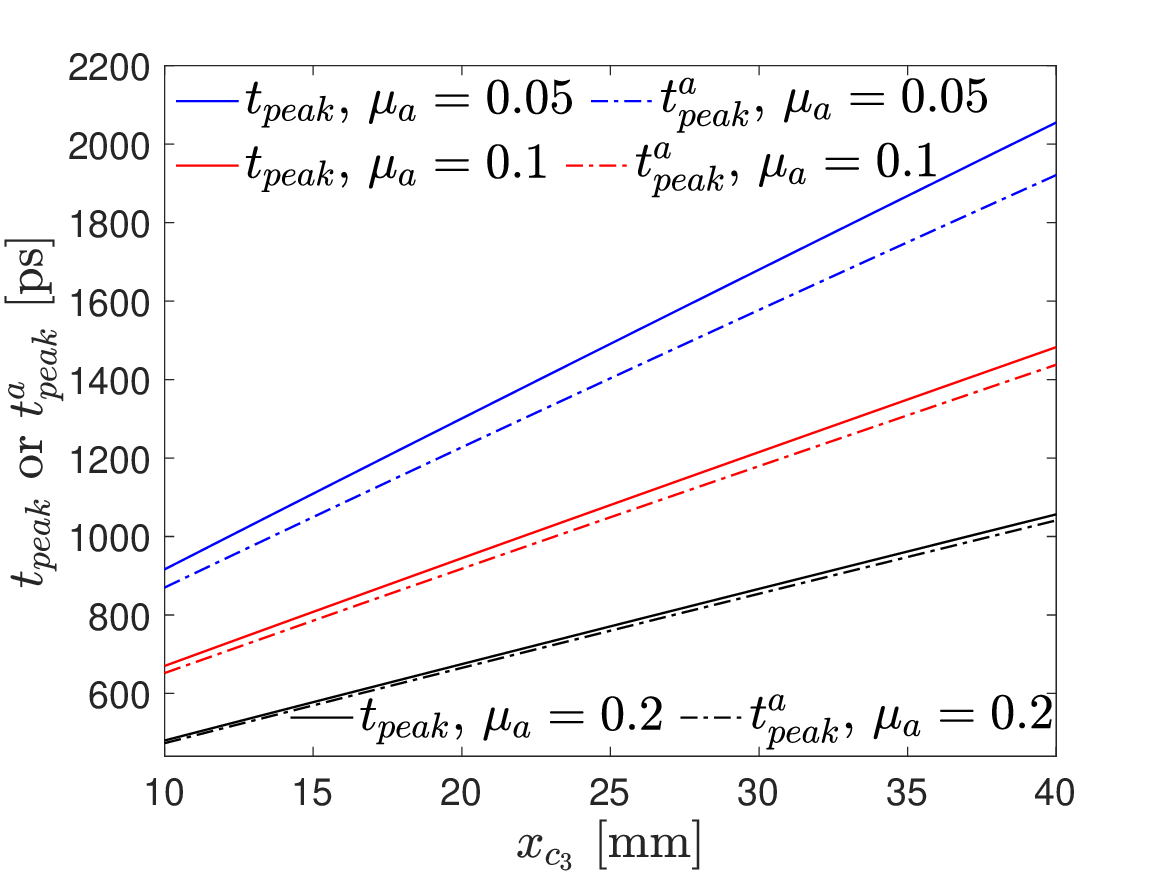}\\
(d) & (e) & (f) \\
\includegraphics[width=0.33\textwidth]{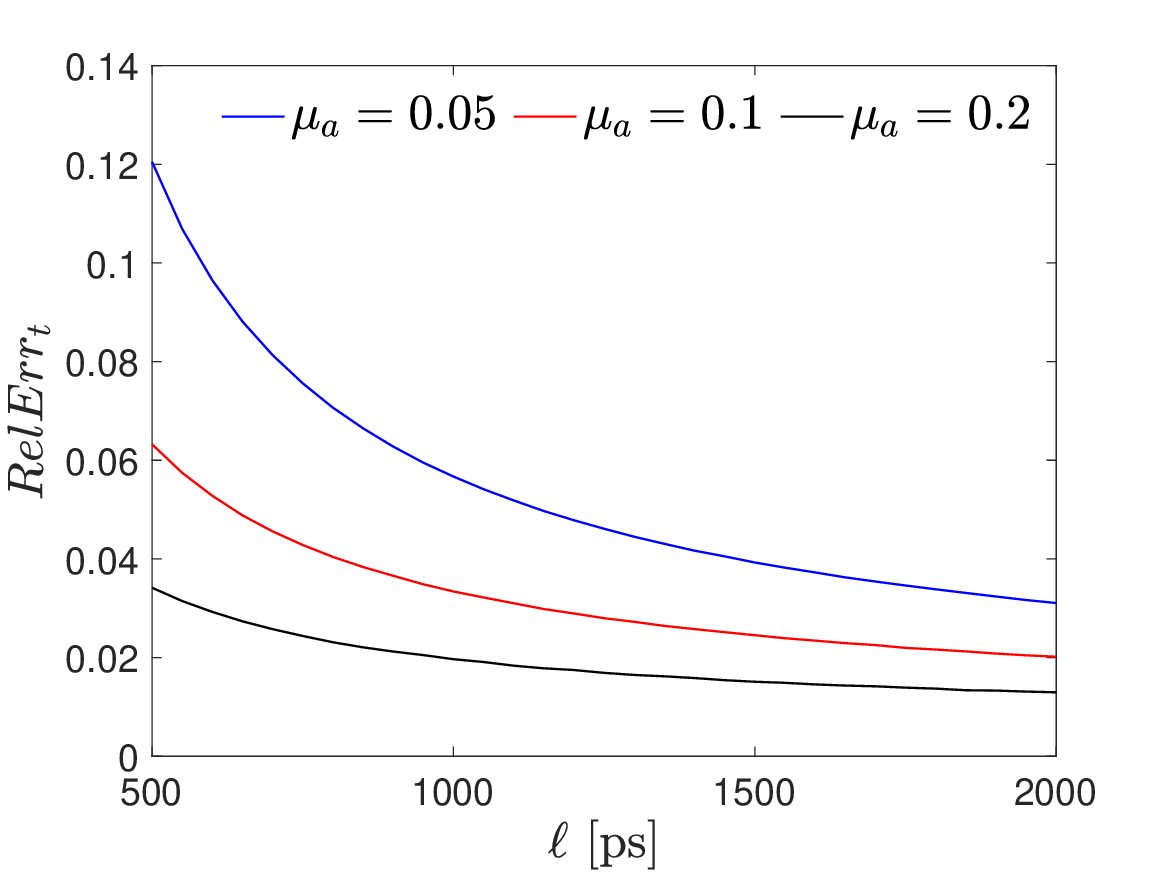}
&\includegraphics[width=0.33\textwidth]{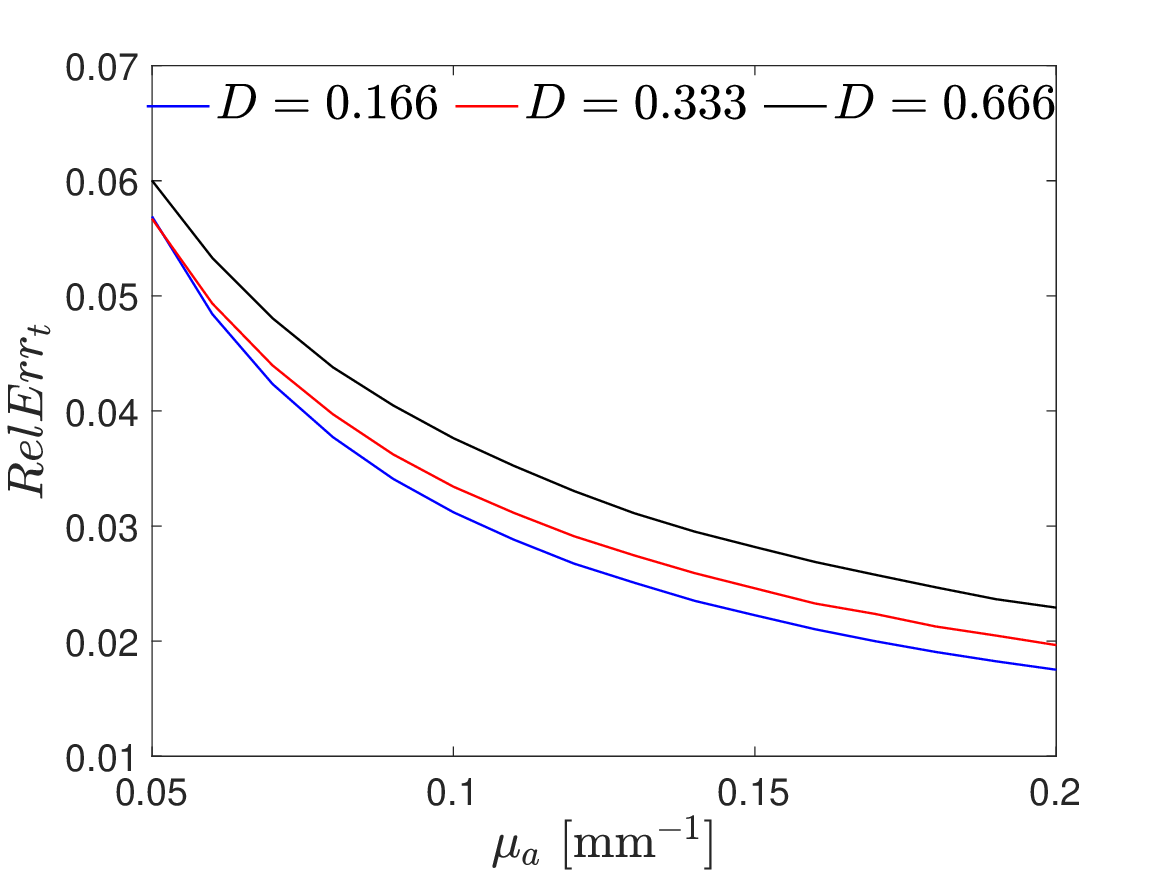}
&\includegraphics[width=0.33\textwidth]{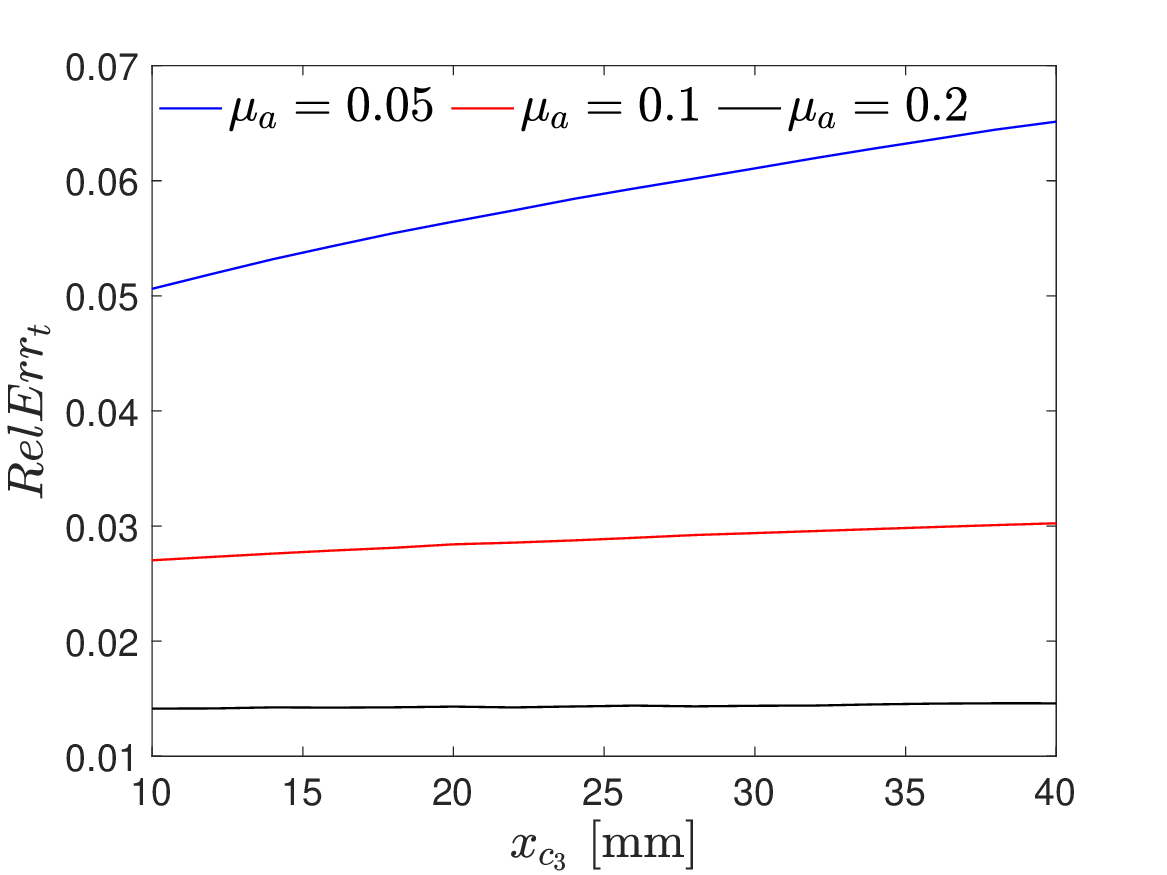}
\end{tabular}
\caption{Peak time, approximate peak time, and relative error for different physical parameters}
\label{fig_tps_diffpara}
\end{figure}

\subsection{Approximate peak time for multiple point targets}
\label{subsec_app_tps}
For multiple point targets $x_c^{(j)},\,j=1,\,2,\,\cdots,\,J$, recall Corollary \ref{cor_asym}.  Let $x_d,\,x_s\in\partial\Omega$ satisfying \eqref{equ_dis_rela_mult} for some $x_c^{(l)},\,l=1,\,2,\cdots,\,J$. Hence, we can define the approximate peak time $t_{peak}^{a,(l)}$ to $x_c^{(l)}$ as the root of the approximate peak time equation 
\begin{equation}\label{tp_formulas}
P^{(l)}(t,\,\lambda^{(l)})=0,
\end{equation}
where $P^{(l)}(t,\,\lambda^{(l)})$ can be defined by \eqref{func_peak} replacing $\lambda$ with $\lambda^{(l)}:= \left(\frac{|\hat{x}_d-x_c^{(l)}|^2+|\hat{x}_s-x_c^{(l)}|^2}{2vD}\right)^{\frac{1}{2}}$. From Theorem \ref{thm_uniq_apppeak}, there exists a unique root of $P^{(l)}(t;\,\lambda^{(l)})$ for $t\in (\lambda^{(l)}k^{-\frac{1}{2}},\,\infty)$. 
As mentioned by Remark \ref{rem_mult_tp}, we can define the unique approximate peak time $t_{peak}^{a,(l)}$ for each $x_c^{(l)},\,l=1,\,2,\,\cdots,\,J$, if $x_d,\,x_s\in\partial\Omega$ satisfy \eqref{equ_dis_rela_mult}. 

Next, we numerically verify the accuracy of each approximate peak time $t_{peak}^{a,(l)},\,l=1,\,2,\cdots,\,J,$ for different S-D pairs.
Let us assume there are two point targets, i.e., $J=2$ in \eqref{mult_targets}, with locations  $x_c^{(1)}=(5,\,10,\,20),\,x_c^{(2)}=(15,\,10,\,20)$. 
Setting the physical parameters as \eqref{phys_para},
the peak time $t_{peak}$ and its approximations $t_{peak}^{a,(l)},\,l=1,\,2,$ can be calculated for any S-D pair by using \eqref{expression of tilde um} and \eqref{tp_formulas}, respectively. Define a set of S-D pairs as
\begin{equation}\label{equ_gridpointSD}
\left\{\{x_d^{(m,n)},\,x_s^{(m,n)}\}:=\{(1+m,\,n,\,0),\,(-1+m,\,n,\,0)\},\;\;m,\,n=0,\,1,\,\cdots,\,20\right\}.
\end{equation}
Figure \ref{fig_Apptps_2020} (a), (b) and (c) plot the values of the peak time $t_{peak}(x_d^{(m,5)},\,x_s^{(m,5)})$ and the approximate peak times $t_{peak}^{a,(1)}(x_d^{(m,5)},\,x_s^{(m,5)}),\,t_{peak}^{a,(2)}(x_d^{(m,5)},\,x_s^{(m,5)})$ at $m=0,\,1,\,\cdots,\,20$, respectively. We find that $t_{peak}^{a,(1)}$ and $t_{peak}^{a,(2)}$ have symmetric shapes, and give a good approximation to $t_{peak}$ for $m<10$ and $m>10$, respectively, since the condition \eqref{equ_dis_rela_mult} is satisfied for $m<10$ and $m>10$ to $x_c^{(1)}$ and $x_c^{(2)}$, respectively. In other words, as long as the chosen S-D pair satisfies \eqref{equ_dis_rela_mult} to some $x_c^{(l)},\,l=1,\,2,\,\cdots,J$, the defined  $t_{peak}^{a,(l)}$ can approximate $t_{peak}$ very well. 

Let us define the approximate peak time for these two point targets as follows:
\begin{equation}\label{appro_peak_dis}
t_{peak}^a(x_d^{(m,n)},x_d^{(m,n)}):=
\begin{cases}
    t_{peak}^{a,(1)}(x_d^{(m,n)},x_d^{(m,n)}),\; & m=0,\,1,\,\cdots,\,10,\,n=0,\,1,\,\cdots,\,20,\\
    t_{peak}^{a,(2)}(x_d^{(m,n)},x_d^{(m,n)}),\; &m=11,\,12,\,\cdots,\,20,\,n=0,\,1,\,\cdots,\,20.
\end{cases}
\end{equation}
which is plotted for all S-D pairs \eqref{equ_gridpointSD} in Figure \ref{fig_Apptps_2020} (d).  Comparing with $t_{peak}$ shown in Figure \ref{fig_Apptps_2020} (e), they have the same shape. The relative error between $t_{peak}$ and $t_{peak}^{a}$ defined by \eqref{appro_peak_dis} is shown in Figure \ref{fig_Apptps_2020} (f), which implies the accuracy of $t_{peak}^a$ defined by \eqref{appro_peak_dis}.

\begin{figure}[htp]
\centering
\begin{tabular}{lll}
(a) & (b) & (c) \\
\includegraphics[width=0.33\textwidth]{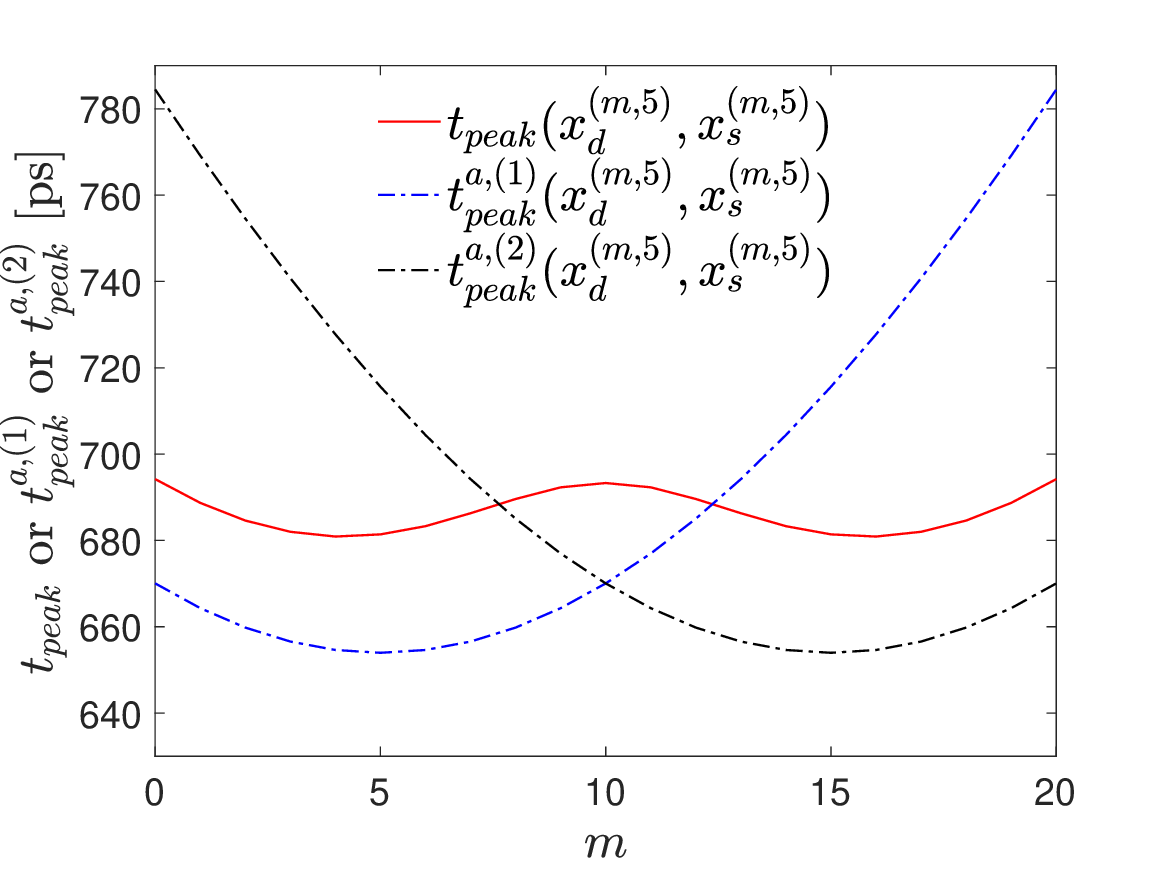}
& \includegraphics[width=0.33\textwidth]{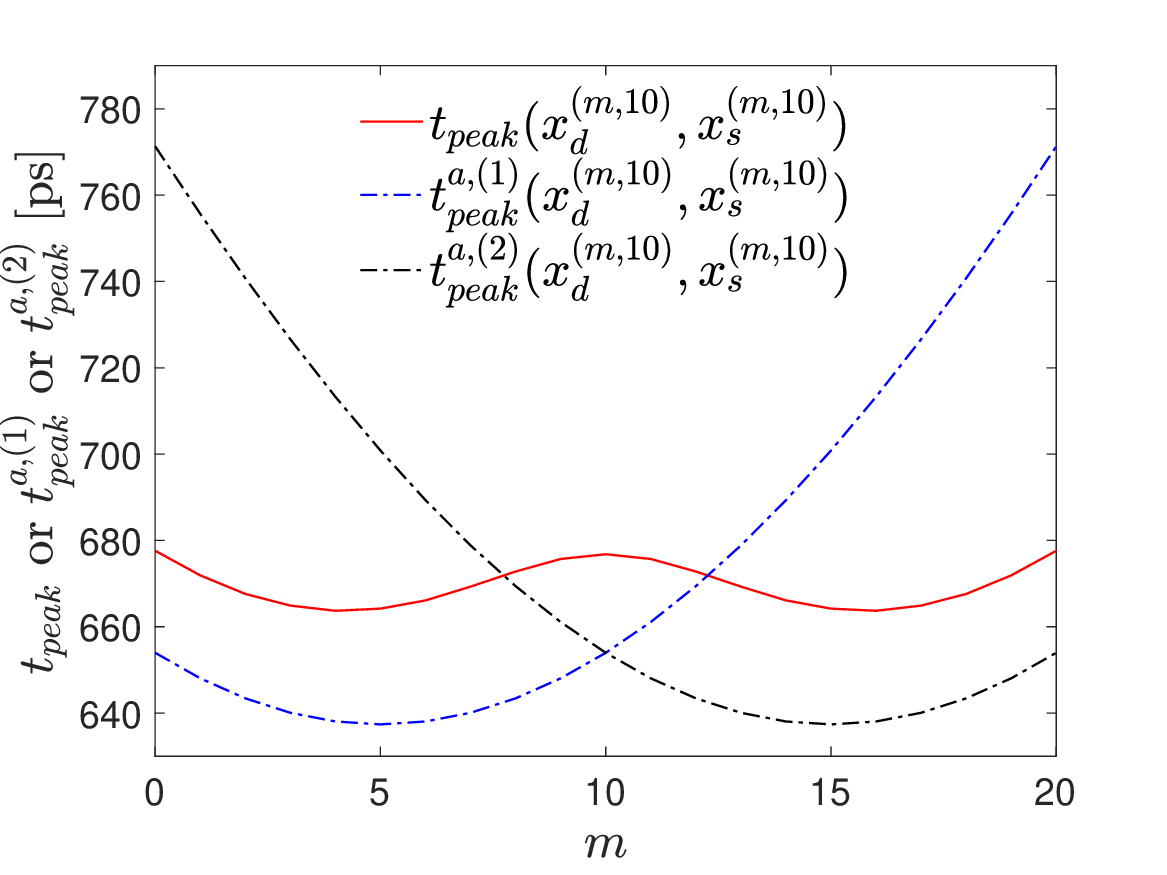}
& \includegraphics[width=0.33\textwidth]{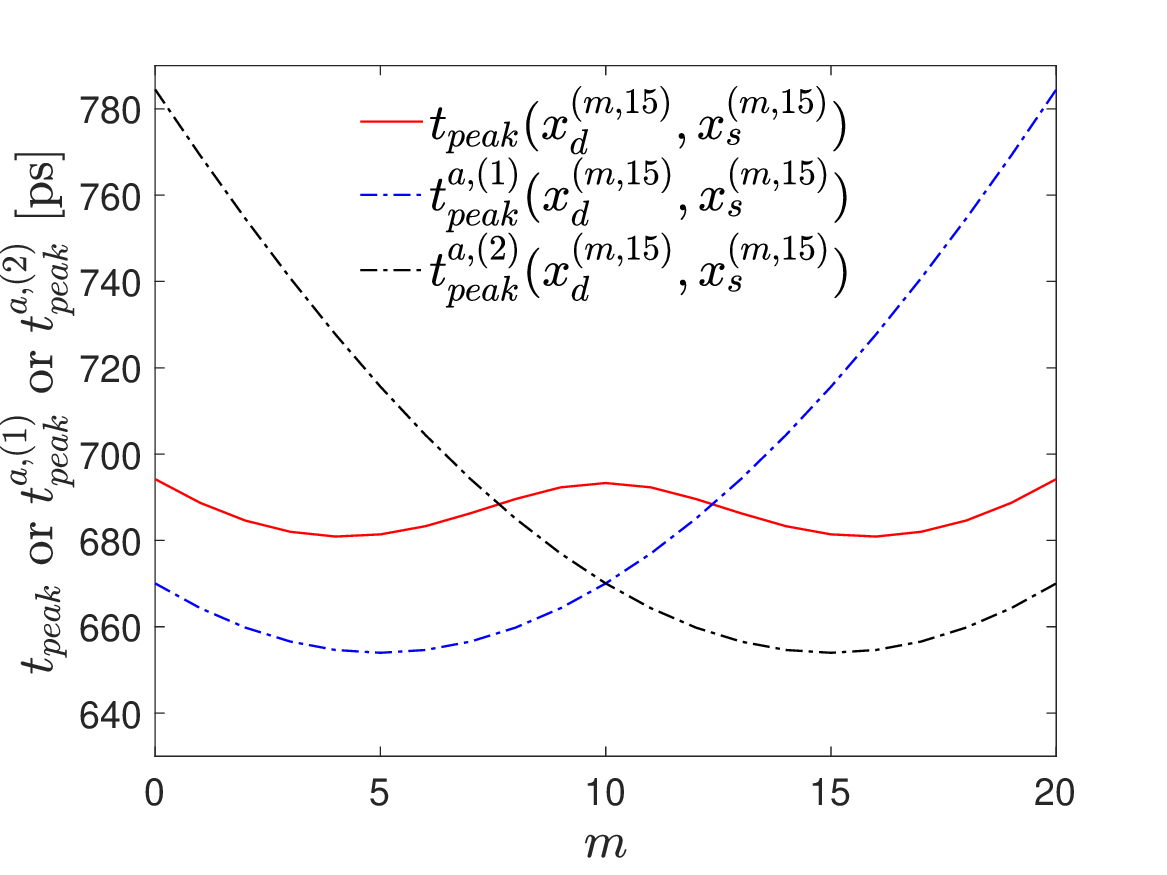}
\end{tabular}
\begin{tabular}{lll}
(d) & (e) & (f) \\
\includegraphics[width=0.33\textwidth]{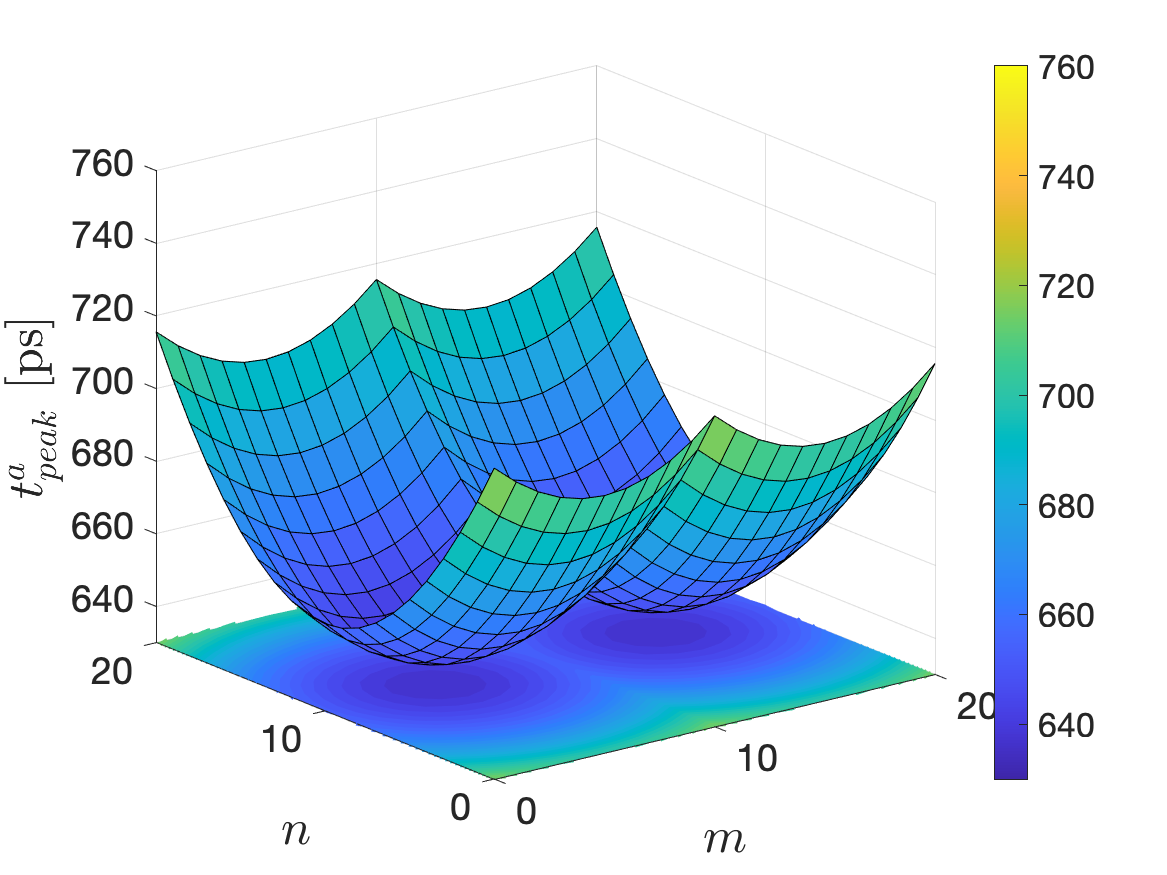}
&\includegraphics[width=0.33\textwidth]{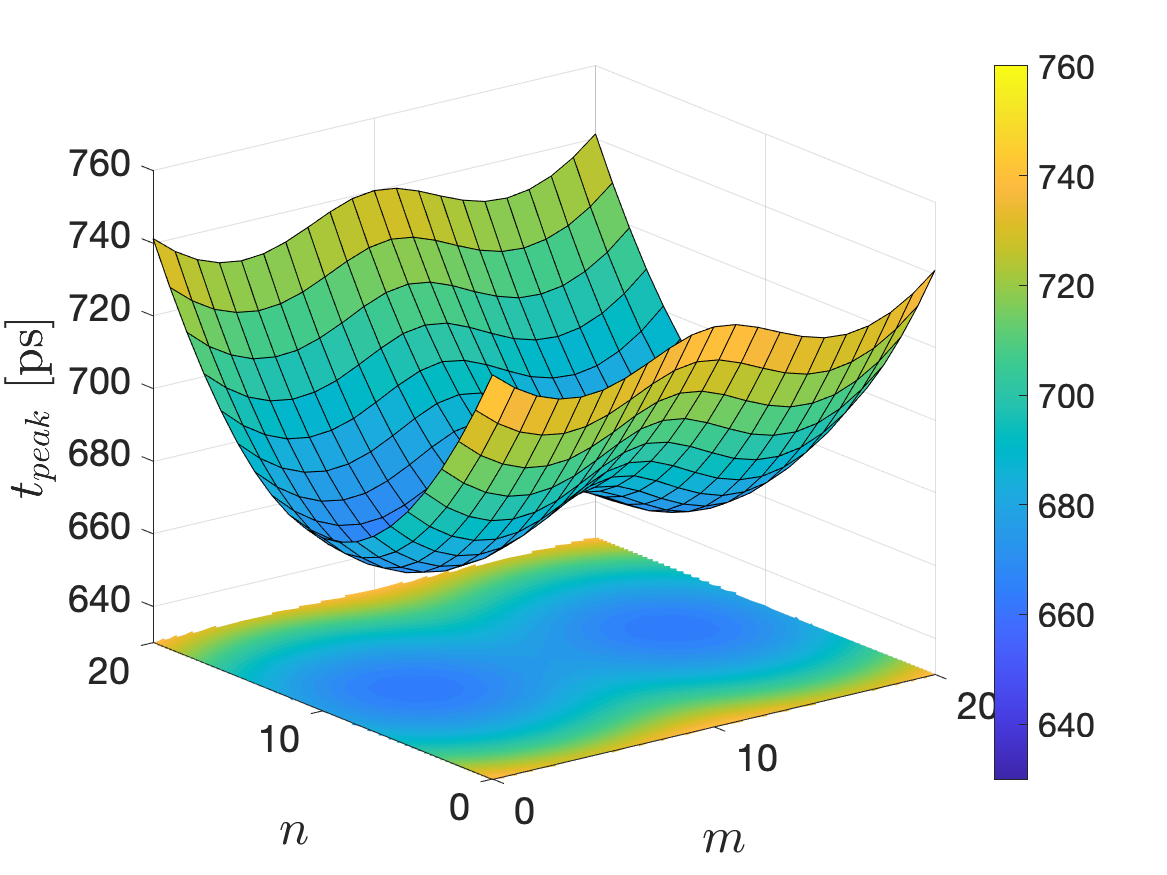}
&\includegraphics[width=0.33\textwidth]{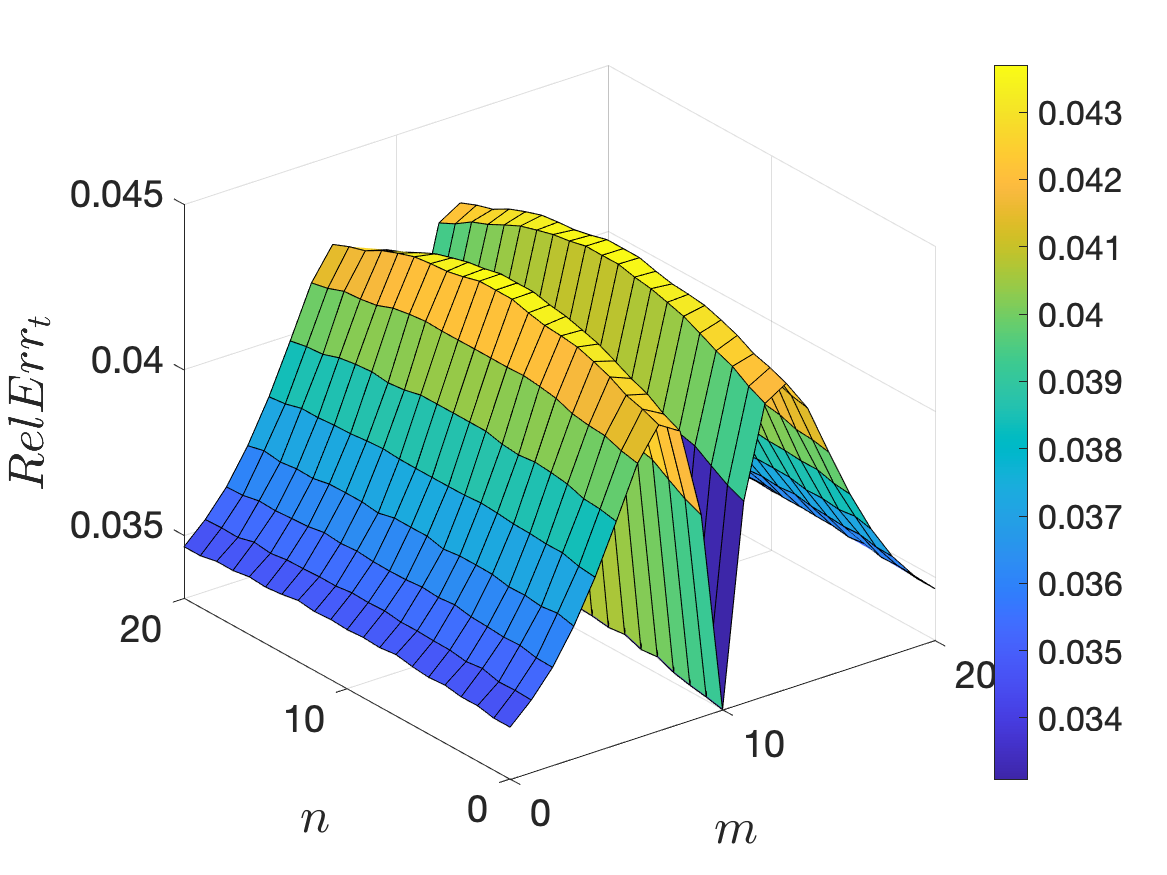}
\end{tabular}
\caption{Peak time, approximate peak time, and relative error for different S-D pairs}
\label{fig_Apptps_2020}
\end{figure}

\section{Reconstruction algorithm}
\label{sec_algo}
%In this section, we introduce a bisection reconstruction algorithm to reconstruct $x_c$ from several measured peak times. More precisely, by using the mentioned order relation and symmetry, we first reconstruct $(x_{c_1},\,x_{c_2})$ by the bisection method taking the measured peak times as an index. Then, the depth $x_{c_3}$ can be reconstructed by solving the approximate peak time equation \eqref{peak time eq}.
In this section, we study the mathematical properties of the approximate peak time and then numerically verify them.
Based on the properties of the peak time, we develop a bisection reconstruction algorithm and a boundary-scan reconstruction algorithm for single point and multiple point targets, respectively, both of which include two stages. In the first stage, we reconstruct the first two coordinates of each target by using the properties of peak time. Then, the third coordinate of each target is reconstructed by solving the approximate peak time equation.

\subsection{Properties of peak time related to target location}

In this subsection, we rigorously prove some properties of the approximate peak time related to the distance between the target location and the S-D pair, which are also numerically verified to the peak time. 
\medskip

%We first show the properties of the approximate peak time, and then numerically verify the accuracy of the approximate peak time compared with the peak time. 

%We show the uniqueness, some properties of the approximate peak time as follows. 

%Before showing the properties of $t_{peak}^a$ related to $\lambda$,
We first state the following monotonicity of $P(\lambda;\,t)$ defined by \eqref{func_peak} with respect to $\lambda$.

\begin{lemma}\label{lem_mono_resp_lamb}
For $t>0$, $P(\lambda;\,t)$ is a monotonically increasing function for $\lambda\in(0,\,tk^{\frac{1}{2}})$.
\end{lemma}
\begin{proof}
Let us examine the monotonicity of $P(\lambda;\,t)$ by considering
%\begin{equation}
$$\partial_\lambda P=\left( 1+\frac{2\lambda(\sqrt{k}t-\lambda)}{t}\right)e^{-\frac{(\sqrt{k}t-\lambda)^2}{t}}.$$
%\end{equation}
Due to the assumption $\lambda<t k^{\frac{1}{2}}$ given in Theorem \ref{solbehaviortheorem2}, we have $\partial_\lambda P>0$ for $\lambda\in(0,\,tk^{\frac{1}{2}})$. 
\end{proof}

By showing the properties of $t_{peak}^a$ related to $\lambda$, we can obtain the solvability of $(x_{c_1},\,x_{c_2})$ under additional assumptions on S-D pairs as follows.
\begin{theorem}\label{thm_orderpeak}
For S-D pairs $x_d^{(n)},\,x_s^{(n)}\in \partial\Omega,\,n=1,\,2,$ we have the following equivalence 
\begin{equation}\label{equ_equiva_dis_t}
t_{peak}^{a,(1)}\geq t_{peak}^{a,(2)} \Longleftrightarrow \lambda^{(1)}\geq \lambda^{(2)},
\end{equation}
where $t_{peak}^{a,(n)}:=t_{peak}^a(x_d^{(n)},\,x_s^{(n)};\,x_c)$ and $\lambda^{(n)}:=\lambda(x_d^{(n)},\,x_s^{(n)};\,x_c)$, $n=1,\,2$.
If we assume that S-D pairs satisfy
\begin{equation}\label{equ_SDdisequal}
|x_d^{(1)}-x_s^{(1)}|=|x_d^{(2)}-x_s^{(2)}|,
\end{equation}
$t_{peak}^{a,(2)}$ attains its unique minimum when $x_d^{(2)}$ and $x_s^{(2)}$ satisfy 
\begin{equation}\label{equ_centerSD}
x_{c_1}=\frac{x_{d_1}^{(2)}+x_{s_1}^{(2)}}{2},\;\;x_{c_2}=\frac{x_{d_2}^{(2)}+x_{s_2}^{(2)}}{2}.
\end{equation}
where $x_{d_j}^{(2)}$ and $x_{s_j}^{(2)},\,j=1,\,2,$ denote the $j$-th coordinate of $x_d^{(2)}$ and $x_s^{(2)}$, respectively. 

If we further assume that S-D pairs satisfy
\begin{equation}\label{SDs_samedirec1}
x_{s_1}^{(1)}-x_{s_1}^{(2)}=x_{d_1}^{(1)}-x_{d_1}^{(2)},\;\;x_{s_2}^{(1)}=x_{s_2}^{(2)},\;\;x_{d_2}^{(1)}=x_{d_2}^{(2)},
\end{equation}
then, $t_{peak}^{a,(1)}=t_{peak}^{a,(2)}$ implies
\begin{equation}\label{equ_solo_xc1}
x_{c_1}=\frac{x_{s_1}^{(1)}+x_{d_1}^{(2)}}{2}.
\end{equation}
Similarly, if %we let $x_d^{(n)},\,x_s^{(n)}\in\partial\Omega,\,n=1,\,2,$ satisfy
\begin{equation}\label{SDs_samedirec2}
x_{s_2}^{(1)}-x_{s_2}^{(2)}=x_{d_2}^{(1)}-x_{d_2}^{(2)},\;\;x_{s_1}^{(1)}=x_{s_1}^{(2)},\;\;x_{d_1}^{(1)}=x_{d_1}^{(2)},
\end{equation}
then, $t_{peak}^{a,(1)}=t_{peak}^{a,(2)}$ implies
\begin{equation}\label{equ_solo_xc2}
x_{c_2}=\frac{x_{s_2}^{(1)}+x_{d_2}^{(2)}}{2}.
\end{equation}
\end{theorem}

\begin{proof}
%For simplicity, we write 
%\begin{equation}\label{equ_sim_tp}
%\lambda^{(n)}:=\lambda(x_d^{(n)},\,x_s^{(n)};\,x_c),\;\;t_{peak}^{a,(n)}:=t_{peak}^a(x_d^{(n)},\,x_s^{(n)};\,x_c),\;\;n=1,\,2.
%\end{equation}
By Theorem \ref{thm_uniq_apppeak} and for $\lambda^{(n)}$, there exists a unique $t_{peak}^{a,(n)}$ satisfying $\lambda^{(n)}<t_{peak}^{a,(n)}k^{\frac{1}{2}}$ such that 
\begin{equation}\label{equ_Pt}
P(t_{peak}^{a,(n)};\,\lambda^{(n)})=0,\;\;n=1,\,2.
\end{equation}
We divide the proof of the equivalence \eqref{equ_equiva_dis_t} into two steps. We first prove the equivalence $t_{peak}^{a,(1)}> t_{peak}^{a,(2)}\Leftrightarrow \lambda^{(1)}> \lambda^{(2)}$.  Suppose that $t_{peak}^{a,(1)}>t_{peak}^{a,(2)}$. Hence, we have 
\begin{equation}\label{equ_rela_lamb2_t12}
\lambda^{(2)}<t_{peak}^{a,(2)}k^{\frac{1}{2}}<t_{peak}^{a,(1)}k^{\frac{1}{2}}.
\end{equation}
By the proof of Theorem \ref{thm_uniq_apppeak},
we have
\begin{equation}\label{charactrization_approx_peak}
t_{peak}^{a,(2)}=\overline{t_{peak}^{a,(2)}}
=\underline{t_{peak}^{a,(2)}},
\end{equation}
where
\begin{equation}\label{peakform_larg0}
\begin{array}{l}
\overline{t_{peak}^{a,(2)}}:=\sup\{t>0:\,P(t';\,\lambda^{(2)})>0\,\,(\lambda^{(2)} k^{-\frac{1}{2}}<t'<t)\},\\
\underline{t_{peak}^{a,(2)}}:=\inf\{t>0:\,P(t';\,\lambda^{(2)})<0\,\,(t'>t)\}.
\end{array}
\end{equation}
Then, $t_{peak}^{a,(1)}>t_{peak}^{a,(2)}$ implies
\begin{equation}\label{equ_Pt2Pt1_lamb2}
0=P(t_{peak}^{a,(2)};\,\lambda^{(2)})>P(t_{peak}^{a,(1)};\,\lambda^{(2)}).
\end{equation}
Combining this with $P(t_{peak}^{a,(1)};\,\lambda^{(1)})=0$, we have
$$P(t_{peak}^{a,(1)};\,\lambda^{(1)})>P(t_{peak}^{a,(1)};\,\lambda^{(2)}).$$
By Lemma \ref{lem_mono_resp_lamb} together with \eqref{equ_rela_lamb2_t12}, we obtain $\lambda^{(1)}>\lambda^{(2)}$.% i.e., \eqref{order_dist}.

Conversely, we assume that $\lambda^{(1)}>\lambda^{(2)}$, which also gives the condition
\begin{equation*}
\lambda^{(2)}<\lambda^{(1)}<t_{peak}^{a,(1)}k^{\frac{1}{2}}.
\end{equation*}
Combining this with Lemma \ref{lem_mono_resp_lamb}, we have
\begin{equation*}
    0=P(t_{peak}^{a,(1)},\,\lambda^{(1)})>P(t_{peak}^{a,(1)},\,\lambda^{(2)}).
\end{equation*}
Since $P(t_{peak}^{a,(2)},\,\lambda^{(2)})=0$, we arrive at \eqref{equ_Pt2Pt1_lamb2}. Due to 
$$\lambda^{(2)}<\lambda^{(1)}<t_{peak}^{a,(1)}k^{\frac{1}{2}}\;\;{\rm and}\;\;\lambda^{(2)}<t_{peak}^{a,(2)}k^{\frac{1}{2}},$$
the relation \eqref{equ_Pt2Pt1_lamb2} implies $t_{peak}^{a,(1)}>t_{peak}^{a,(2)}$ from the proof of Theorem \ref{thm_uniq_apppeak}.

Next, we prove the equivalence $t_{peak}^{a,(1)}=t_{peak}^{a,(2)}\Leftrightarrow \lambda^{(1)}=\lambda^{(2)}$. Assuming that $\lambda^{(1)}=\lambda^{(2)}$, it is obvious that $t_{peak}^{a,(1)}=t_{peak}^{a,(2)}$ by using the uniqueness of the approximate peak time. On the other hand, we assume that $t_{peak}^{a,(1)}=t_{peak}^{a,(2)}$. From \eqref{equ_Pt},  we have 
\begin{equation}\label{equ_tp1=tp2}
t_{peak}^{a,(1)}(\ln\lambda^{(1)}-\ln\lambda^{(2)})=(\sqrt{k}t_{peak}^{a,(1)}-\lambda^{(1)})^2-(\sqrt{k}t_{peak}^{a,(1)}-\lambda^{(2)})^2.
\end{equation}
For fixed $\lambda^{(1)}$, the left hand side of \eqref{equ_tp1=tp2} increases as $\lambda^{(2)}\in (0,\,t_{peak}^{a,(1)}k^{\frac{1}{2}})$ increases, monotonically. Similarly, the right-hand side of \eqref{equ_tp1=tp2} decreases monotonically. Hence, there must be $\lambda^{(1)}=\lambda^{(2)}$ in \eqref{equ_tp1=tp2}.

The value $\lambda^{(2)}$ has a unique smallest value whenever $x_d^{(2)}$ and $x_s^{(2)}$ satisfy the condition \eqref{equ_centerSD}. That is, \eqref{equ_centerSD} gives the condition of taking the minimal approximate peak time. The uniqueness has been proven in Theorem \ref{thm_uniq_apppeak}.

With an additional assumption \eqref{SDs_samedirec1}, let 
$L:=x_{s_1}^{(1)}-x_{s_1}^{(2)}=x_{d_1}^{(1)}-x_{d_1}^{(2)}.$
We have proved the fact that $t_{peak}^{a,(1)}=t_{peak}^{a,(2)}$ implies $\lambda^{(1)}=\lambda^{(2)}$, which gives 
\begin{equation}\label{equ_same_dis}
(x_{d_1}^{(1)}-x_{c_1})^2+(x_{s_1}^{(1)}-x_{c_1})^2=(x_{d_1}^{(2)}-x_{c_1})^2+(x_{s_1}^{(2)}-x_{c_1})^2.
\end{equation}
Substituting $x_{d_1}^{(1)}=L+x_{d_1}^{(2)}$ and $x_{s_1}^{(1)}=L+x_{s_1}^{(2)}$ into above equation, there is a unique solution 
$$x_{c_1}=\frac{x_{d_1}^{(2)}+x_{s_1}^{(2)}+L}{2}=\frac{x_{s_1}^{(1)}+x_{d_1}^{(2)}}{2}.$$
Under the assumption \eqref{SDs_samedirec2}, $x_{c_2}$ can be uniquely determined from \eqref{equ_same_dis}.
\end{proof}

In the following, we numerically verify that the peak time possesses the same properties given in Theorem \ref{thm_orderpeak}.
%In this verification, we show the results for different S-D pairs. 

Let $x_c=(10,\,10,\,20)$ and the physical parameters be \eqref{phys_para}.
We compute the peak time, the approximate peak time, and the relative error for a set of S-D pairs defined as
\begin{align}\label{SD8_index}
\left\{\{x_d^{(m,n)},\,x_s^{(m,n)}\}=\left\{(4+m,\,n,\,0),\,(-4+m,\,n,\,0)\right\},\;\;m,\,n=0,\,1,\,\cdots,\,20\right\},
\end{align}
which are plotted in Figure \ref{fig_tp_life_depth20}.
For simplicity, we write 
\begin{align*}
&t_{peak}^{(m,n)}:= t_{peak}(x_d^{(m,n)},\,x_s^{(m,n)};\,x_c),\;\;\;\;t_{peak}^{a,(m,n)}:= t_{peak}^a(x_d^{(m,n)},\,x_s^{(m,n)};\,x_c),\nonumber\\
&\lambda^{(m,n)}:=\lambda(x_d^{(m,n)},\,x_s^{(m,n)};\,x_c),\;\;\;\;
RelErr_t^{(m,n)}:=RelErr_t(x_d^{(m,n)},\,x_s^{(m,n)};\,x_c).
\end{align*}
 Here, the relative error $RelErr_t^{(m,n)}$ has a similar definition given in \eqref{relerr_tp}. We mention that each grid point $(m,\,n)$ in Figure \ref{fig_tp_life_depth20} (a), (b) and (c) corresponds to $t_{peak}^{(m,n)},\,t_{peak}^{a,(m,n)}$ and $RelErr_t^{(m,n)}$, respectively.

\begin{figure}[htp]
\centering
\begin{tabular}{lll}
(a) & (b) & (c)\\
\includegraphics[width=0.31\textwidth]{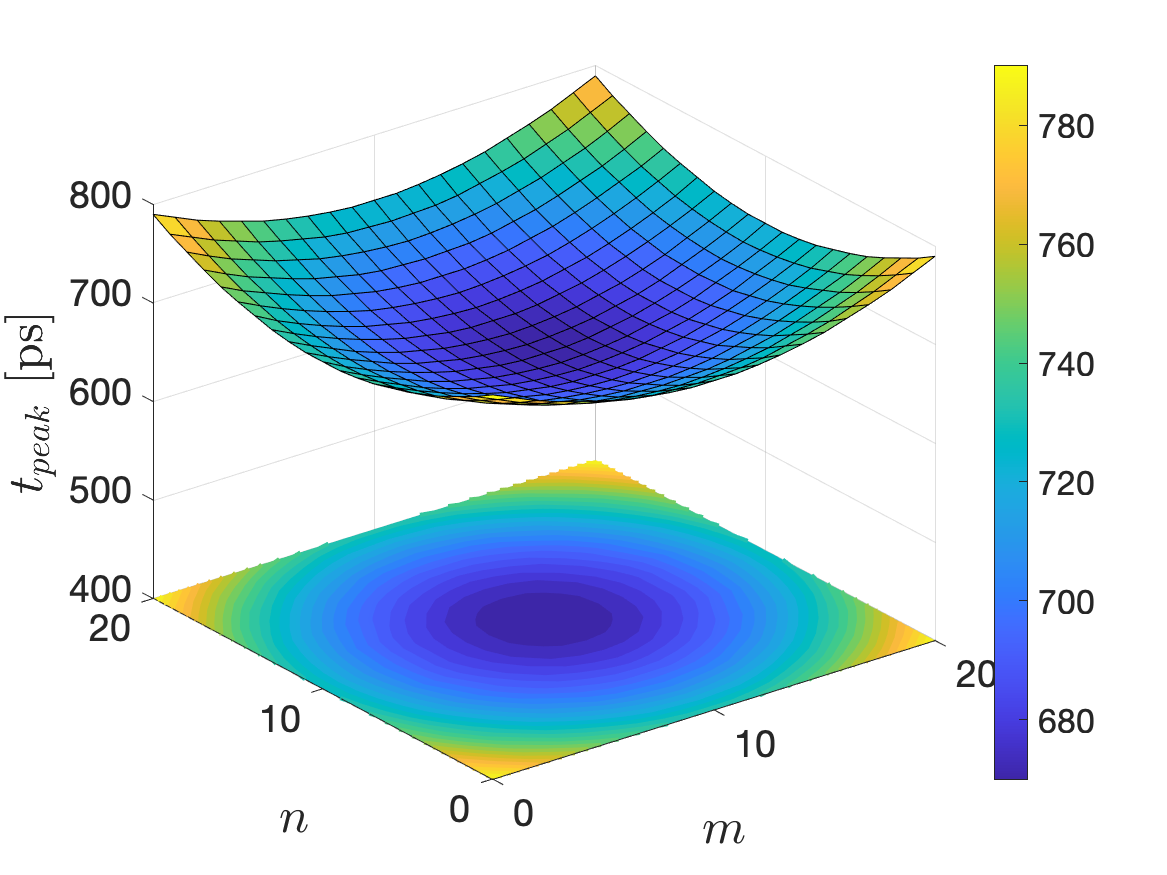}
& \includegraphics[width=0.31\textwidth]{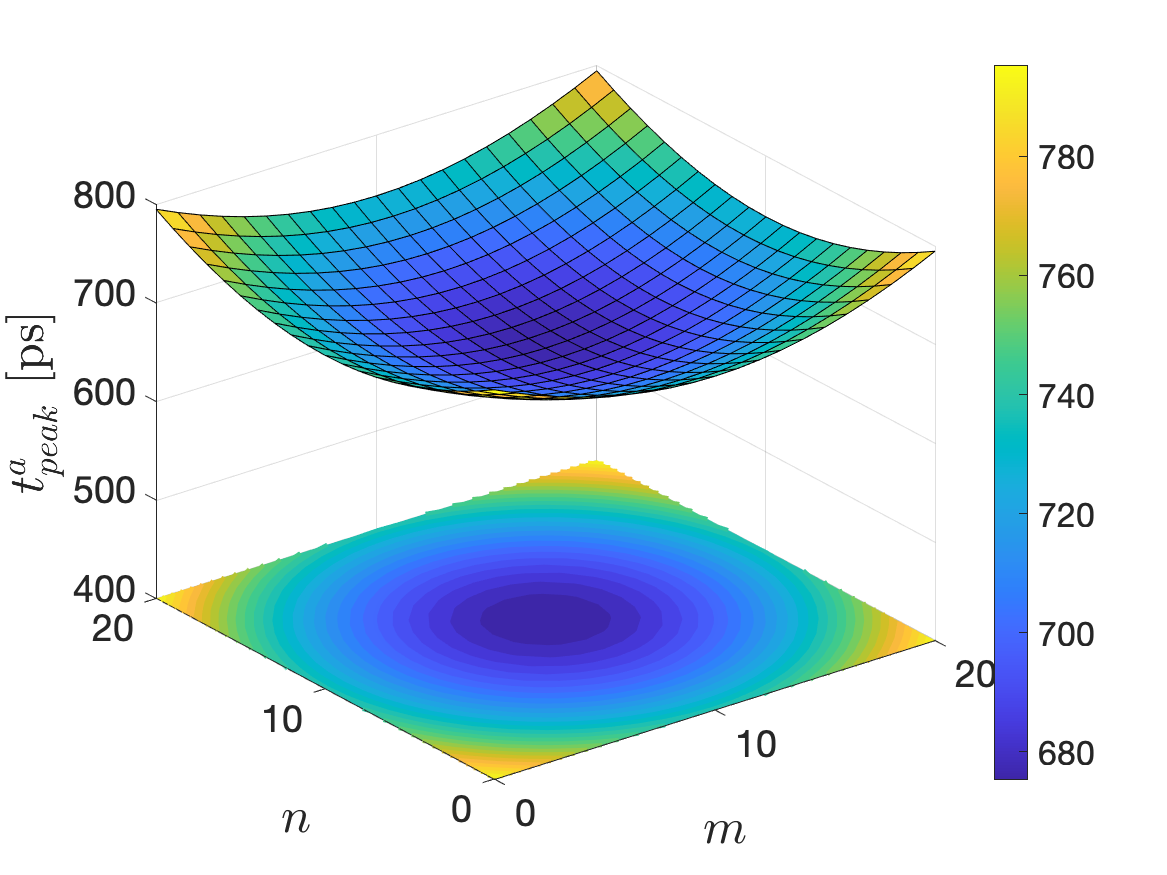}
&\includegraphics[width=0.31\textwidth]{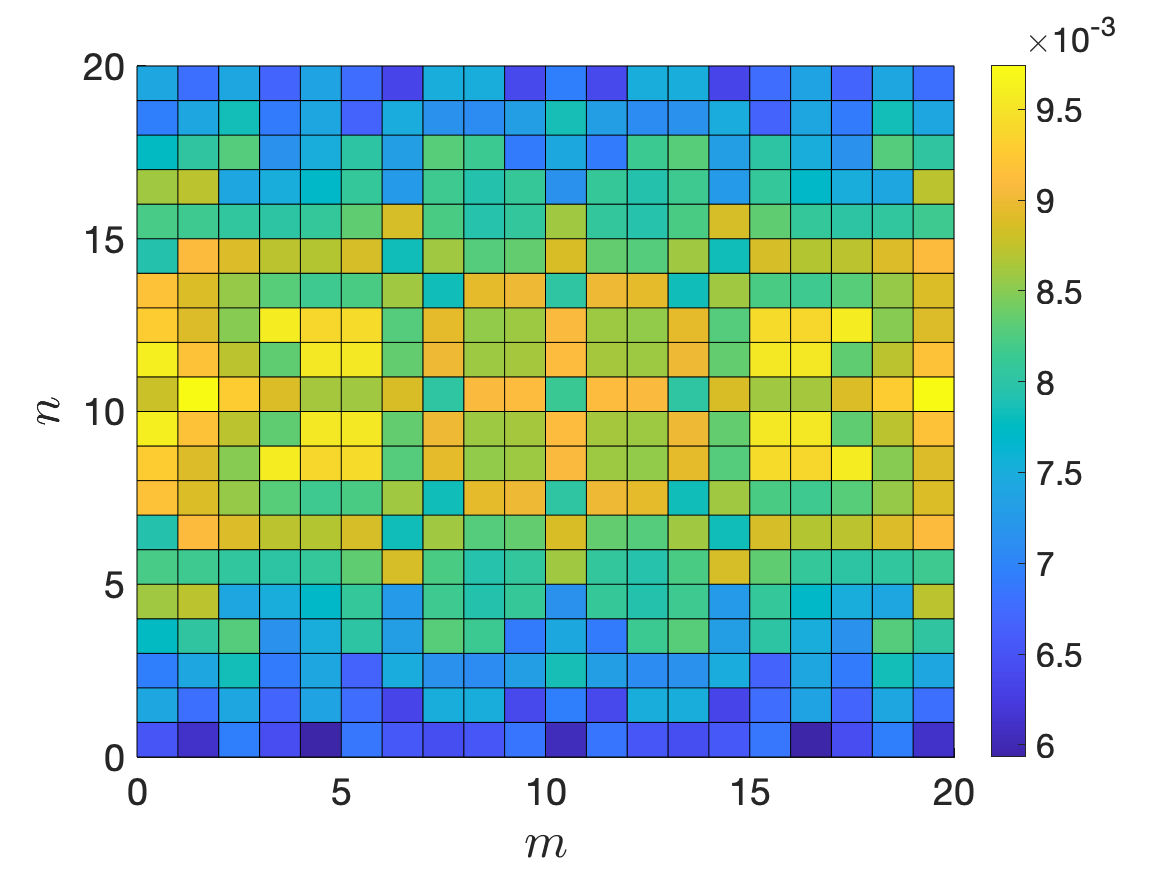}
\end{tabular}
\caption{(a) peak time, (b) approximate peak time, and (c) relative error for S-D pairs defined by \eqref{SD8_index}}
\label{fig_tp_life_depth20}
\end{figure}

We have the following observations of the peak time from the results shown in Figure \ref{fig_tp_life_depth20}:
\begin{itemize}
    \item[{\rm (i)}] Uniqueness: there always exists only one peak time for each S-D pair defined by \eqref{SD8_index}.
    \item[{\rm (ii)}] Order relation: $t_{peak}^{(m,n)}$ decreases as $(m,n)\rightarrow (10,10)$, since $\lambda^{(m,n)}$ decreases as $(m,n)\rightarrow (10,10)$. We obtain the smallest peak time at $\{x_d^{(10,10)},x_s^{(10,10)}\}$, which satisfies $x_{d_1}^{(10,10)}+x_{s_1}^{(10,10)}=2x_{c_1},\,x_{d_2}^{(10,10)}+x_{s_2}^{(10,10)}=2x_{c_2}$, i.e., \eqref{equ_centerSD}.
    \item[{\rm (iii)}] Symmetry: for any $(m_1,n_1)$ and $(m_2,n_2)$ satisfying $\lambda^{(m_1,n_1)}=\lambda^{(m_2,n_2)}$, we observe $t_{peak}^{(m_1,n_1)}=t_{peak}^{(m_2,n_2)}$ from the circular contour shape of the peak times as shown in the bottom plot of Figure \ref{fig_tp_life_depth20} (a). It is easy to verify the solvability of $(x_{c_1},\,x_{c_2})$ given in \eqref{equ_solo_xc1} and \eqref{equ_solo_xc2}, since the center of the circular contour shape is $(x_{c_1},\,x_{c_2})$.
     \item[{\rm (iv)}] Accuracy: the relative error shown in Figure \ref{fig_tp_life_depth20} (c) indicates a good approximation of the approximate peak time.
\end{itemize}

Despite the limited numerical results, it is reasonable to use the approximate peak time equation and the properties of the approximate peak time to formulate a scheme reconstructing the location of the unknown point target. Hence, unless specified, we do not distinguish the peak time and the approximate peak time in the following parts. Here, we remark that we will set S-D pairs satisfying the assumption \eqref{SDs_samedirec1} or \eqref{SDs_samedirec2} in the scheme since \eqref{SDs_samedirec1} and \eqref{SDs_samedirec2} directly imply \eqref{equ_SDdisequal}.

\subsection{Bisection reconstruction algorithm for reconstructing single point target}
For the single point target $x_c$, the properties of the peak time in Theorem \ref{thm_orderpeak} only depend on $(x_{c_1},\,x_{c_2})$. 
%Hence, we introduce a bisection reconstruction algorithm to reconstruct $x_c$ with the following two stages: 
Hence, we divide the reconstruction scheme into the following two stages: 
\begin{description}
  \item[Stage 1:]
  Let $(x_{c_1},\,x_{c_2})$ be inside of a rectangular region of interest $({\rm ROI})$. %, where ${\rm ROI}:=(x_l,\,x_r)\times(x_b,\,x_t)\subset\partial\Omega$. 
  By using the peak time as an index, we use the bisection method to gradually shrink ${\rm ROI}$ and stop by some criterion. We let the center of the updated ${\rm ROI}$ be the reconstruction $(x_{c_1}^{inv},\,x_{c_2}^{inv})$.
%  (see Subsection \ref{subsec_algoxc12} for details).
  \item[Stage 2:]  
Substituting a set of parameters $x_d,\,x_s\in\partial\Omega$, $t=t_{peak}(x_d,\,x_s;\,x_c)$ and $(x_{c_1},\,x_{c_2})=(x_{c_1}^{inv},\,x_{c_2}^{inv})$ into \eqref{func_peak}, the depth $x_{c_3}$ can be reconstructed by finding the positive root of \eqref{peak time eq}.
\end{description}

%\begin{remark}
% We remark that Then, the unique existence of $x_{c_3}^{inv}>0$ can be directly obtained from the unique existence of $\lambda^{inv}$ given in Theorem \ref{thm_unique_xc}.
%\end{remark}

%\subsection{Bisection method for reconstructing $(x_{c_1},\,x_{c_2})$}
%\label{subsec_algoxc12}
Now, we give a detailed explanation of \textbf{Stage 1}. 
Let ${\rm ROI}:=(x_l,\,x_r)\times(x_b,\,x_t)\subset\partial\Omega$.
We define four S-D pairs with $L>0$ as
\begin{equation}\label{4D_SDset}
\begin{array}{l}
\{x_d^{(1)},\,x_s^{(1)}\}:=\left\{(x_l+\frac{L}{2},\,x_b,\,0),\,(x_l-\frac{L}{2},\,x_b,\,0)\right\},
\{x_d^{(2)},\,x_s^{(2)}\}:=\left\{(x_l+\frac{L}{2},\,x_b,\,0),\,(x_r-\frac{L}{2},\,x_b,\,0)\right\},\\
\{x_d^{(3)},\,x_s^{(3)}\}:=\left\{(x_r+\frac{L}{2},\,x_t,\,0),\,(x_r-\frac{L}{2},\,x_t,\,0)\right\},
\{x_d^{(4)},\,x_s^{(4)}\}:=\left\{(x_l+\frac{L}{2},\,x_t,\,0),\,(x_l-\frac{L}{2},\,x_t,\,0)\right\}.
\end{array}
\end{equation}
By comparing the order relation of four corresponding peak times $t_{peak}^{(n)}:=t_{peak}(x_d^{(n)},\,x_d^{(n)};\,x_c),\,n=1,\,2,\,3,\,4$, we halve the current {\rm ROI} and obtain an updated {\rm ROI}, still denoted by ${\rm ROI}=(x_l,\,x_r)\times(x_b,\,x_t)$. Then, we repeat to set four S-D pairs as \eqref{4D_SDset} and compare the order relation of the peak times.  The algorithm breaks out if a criterion is satisfied. 
%$Err_1:=x_r-x_l\leq \epsilon_1,\,Err_2:=x_t-x_b\leq \epsilon_2$ with some $\epsilon_1,\,\epsilon_2>0$. We define the reconstruction $x_{c_1}^{inv}=(x_l+x_r)/2,\,x_{c_2}^{inv}=(x_b+x_t)/2$.
We refer to this scheme as Algorithm \ref{algo_3D}. Here, we will use the symbol \enquote{$\leftarrow$} to represent the update.  

\begin{algorithm}[htp]
\caption{Reconstruct $(x_{c_1},\,x_{c_2})$ by using the bisection method}
\label{algo_3D}
\textbf{Input}: ${\rm ROI}=(x_l,\,x_r)\times(x_b,\,x_t)$, tolerances $\epsilon_1,\,\epsilon_2>0$ and distance $L$.
\begin{description}
\item[Step 1] Define four S-D pairs by \eqref{4D_SDset}.

\item[Step 2] Compute $t_{peak}^{(n)},\,n=1,\,2,\,3,\,4$, and update ${\rm ROI}$.
\begin{description}

\item[1a)]\,If $t_{peak}^{(1)}=t_{peak}^{(2)}=t_{peak}^{(3)}
    =t_{peak}^{(4)}$, break the loop and output $x_{c_1}^{inv}=\frac{x_l+x_r}{2},\,x_{c_2}^{inv}=\frac{x_b+x_t}{2}$.

\item[2a)]\,If $t_{peak}^{(1)}=t_{peak}^{(2)}<t_{peak}^{(3)}
    =t_{peak}^{(4)}$, output $x_{c_1}^{inv}=\frac{x_l+x_r}{2}$,  update $x_t\leftarrow\frac{x_b+x_t}{2}$, ${\rm ROI}\leftarrow(x_b,\,x_t)$, break the loop and go to Algorithm \ref{algo_2D}.

\item[2b)]\,If $t_{peak}^{(1)}=t_{peak}^{(2)}>t_{peak}^{(3)}
    =t_{peak}^{(4)}$, output $x_{c_1}^{inv}=\frac{x_l+x_r}{2}$, update $x_b\leftarrow\frac{x_b+x_t}{2}$, ${\rm ROI}\leftarrow(x_b,\,x_t)$, break the loop and go to Algorithm \ref{algo_2D}.

\item[2c)]\,If $t_{peak}^{(1)}=t_{peak}^{(4)}<t_{peak}^{(2)}
    =t_{peak}^{(3)}$, output $x_{c_2}^{inv}=\frac{x_b+x_t}{2}$, update $x_r\leftarrow\frac{x_l+x_r}{2}$, ${\rm ROI}\leftarrow(x_l,\,x_r)$, break the loop and go to Algorithm \ref{algo_2D}.

\item[2d)]\,If $t_{peak}^{(1)}=t_{peak}^{(4)}>t_{peak}^{(2)}
    =t_{peak}^{(3)}$, output $x_{c_2}^{inv}=\frac{x_b+x_t}{2}$,  update $x_l\leftarrow\frac{x_l+x_r}{2}$, ${\rm ROI}\leftarrow(x_l,\,x_r)$, break the loop and go to Algorithm \ref{algo_2D}.

\item[3a)]\,If $t_{peak}^{(1)}<\{t_{peak}^{(2)},\,t_{peak}^{(3)},\,t_{peak}^{(4)}\}$,
    update $x_r\leftarrow\frac{x_l+x_r}{2},\, x_t\leftarrow\frac{x_b+x_t}{2},\, {\rm ROI}\leftarrow(x_l,\,x_r)\times(x_b,\,x_t)$.

\item[3b)]\,If $t_{peak}^{(2)}<\{t_{peak}^{(1)},\,t_{peak}^{(2)},\,t_{peak}^{(4)}\}$,
    update $x_l\leftarrow\frac{x_l+x_r}{2},\, x_t\leftarrow\frac{x_b+x_t}{2},\, {\rm ROI}\leftarrow(x_l,\,x_r)\times(x_b,\,x_t)$.

\item[3c)]\,If $t_{peak}^{(3)}<\{t_{peak}^{(1)},\,t_{peak}^{(2)},\,t_{peak}^{(4)}\}$,
    update $x_l\leftarrow\frac{x_l+x_r}{2},\, x_b\leftarrow\frac{x_b+x_t}{2},\, {\rm ROI}\leftarrow(x_l,\,x_r)\times(x_b,\,x_t)$.

\item[3d)]\,If $t_{peak}^{(4)}<\{t_{peak}^{(1)},\,t_{peak}^{(2)},\,t_{peak}^{(3)}\}$,
    update $x_r\leftarrow\frac{x_l+x_r}{2},\, x_b\leftarrow\frac{x_b+x_t}{2},\, {\rm ROI}\leftarrow(x_l,\,x_r)\times(x_b,\,x_t)$.
\end{description}
\item[Step 3] Compute $Err_1:=x_r-x_l$, $Err_2:=x_t-x_b$ and check stop criterion.

If $Err_1\leq \epsilon_1,\,Err_2\leq \epsilon_2$, stop and output $x_{c_1}^{inv}=\frac{x_l+x_r}{2},\,x_{c_2}^{inv}=\frac{x_b+x_t}{2}$.

If $Err_1> \epsilon_1,\,Err_2\leq \epsilon_2$, output $x_{c_2}^{inv}=\frac{x_b+x_t}{2}$, go to Algorithm \ref{algo_2D} with ${\rm ROI}:=(x_l,\,x_r)$.

If $Err_1\leq \epsilon_1,\,Err_2> \epsilon_2$, output $x_{c_1}^{inv}=\frac{x_l+x_r}{2}$, go to Algorithm \ref{algo_2D} with ${\rm ROI}:=(x_b,\,x_t)$.

If $Err_1>\epsilon_1,\,Err_2> \epsilon_2$, go to \bf{Step 1}.
\end{description}
\end{algorithm}

\begin{algorithm}[htp]
\caption{Reconstruct $x_{c_1}$ or $x_{c_2}$ after the dimensionality reduction}
\label{algo_2D}
\textbf{Input:} Reconstruction $x_{c_2}^{inv}$, {\rm ROI}=$(x_l,\,x_r)$, tolerance $\epsilon_1$ and distance $L$.

\begin{description}
\item[Step 1] Define S-D pairs $\{x_d^{(n)},\,x_s^{(n)}\},\,n=1,\,2,$ by \eqref{4D_SDset}.

\item[Step 2] Compare $t_{peak}^{(1)}$ and $t_{peak}^{(2)}$ and update ${\rm ROI}$.

If $t_{peak}^{(1)}=t_{peak}^{(2)}$, break the loop and output $x_{c_1}^{inv}=\frac{x_l+x_r}{2}$.

If $t_{peak}^{(1)}<t_{peak}^{(2)}$, update
$x_r\leftarrow\frac{x_l+x_r}{2},\, {\rm ROI}\leftarrow(x_l,\,x_r)$.

If $t_{peak}^{(1)}>t_{peak}^{(2)}$, update
$x_l\leftarrow\frac{x_l+x_r}{2},\, {\rm ROI}\leftarrow(x_l,\,x_r)$.

\item[Step 3] Compute $Err_1=x_r-x_l$ and check stop criterion.

If $Err_1\leq \epsilon_1$, break the loop and output $x_{c_1}^{inv}=\frac{x_l+x_r}{2}$.
 Otherwise, go to \textbf{Step 1}.

\end{description}
(Note: The case of inputting $x_{c_1}^{inv}, \; {\rm ROI}=(x_b,\,x_t),\;\epsilon_2$ and $L$ can be done in the same way.)
\end{algorithm}

It is obvious that there exists a unique reconstruction $(x_{c_1}^{inv},\,x_{c_2}^{inv})$ in \rm{\textbf{Stage 1}}. 
We state the following unique existence of reconstruction $x_{c_3}^{inv}$ from given $t_{peak}$.
\begin{theorem}\label{thm_unique_xc}
Let $x_c^*$ be the true location. Assume that $\ell>\pi^{\frac{1}{2}}k^{-\frac{3}{4}}\left(\lambda^*\right)^{\frac{1}{2}}$ with $\lambda^*:=\lambda(x_d,\,x_s,\,x_c^*)$.
For given peak time $t_{peak}$, there exists a unique positive reconstruction $x_{c_3}^{inv}$ such that $\lambda^{inv}<t_{peak}k^{\frac{1}{2}}$ and $P(\lambda^{inv};\,t_{peak})=0$, where $\lambda^{inv}:=\lambda(x_d,\,x_s,\,x_c^{inv})$ and  $P(\lambda;\,t)$ are defined by \eqref{defk} and \eqref{func_peak}, respectively.
\end{theorem}
\begin{proof}
Due to $x_{c_3}>0$, proving the uniqueness of $x_{c_3}^{inv}$ is equivalent to proving the uniqueness of $\lambda^{inv}$.  From Lemma \ref{lem_mono_resp_lamb}, we obtain that $P(\lambda;\,t_{peak})$ is monotonically increasing for $\lambda\in(0,\,t_{peak}k^{\frac{1}{2}})$. Due to $P(0;\,t_{peak})< 0$, we only need to compute the value $P(\lambda;\,t_{peak})$ at $\lambda=t_{peak}k^{\frac{1}{2}}$,
$$P(t_{peak}k^{\frac{1}{2}};\,t_{peak})=t_{peak}k^{\frac{1}{2}}-\pi^{\frac{1}{2}}\ell^{-1}t_{peak}^{\frac{3}{2}}>t_{peak}k^{\frac{1}{2}}+k^{\frac{3}{4}}(\lambda^*)^{-\frac{1}{2}}t_{peak}^{\frac{3}{2}}>0,$$
which implies the unique existence of $\lambda^{inv}$. From \eqref{defk}, we obtain that there exists a unique $x_{c_3}^{inv}$.
\end{proof}

\subsection{Boundary-scan algorithm for reconstructing multiple point targets}
In this subsection, we propose a boundary-scan method for reconstructing the locations of the multiple point targets $x_c^{(1)},\,x_c^{(2)},\,\cdots,\,x_c^{(J)}$, which are assumed to be well-separated. 
\medskip

%Comparing \eqref{peak time eq} with , 
For each $x_c^{(l)},\,l=1,\,2,\,\cdots,\,J$, the relation between $x_c^{(l)}$ and the approximate peak time $t_{peak}^{a,(l)}$ defined by \eqref{tp_formulas} must be the same as described in Theorem \ref{thm_orderpeak}, if we further assume that the S-D pair $\{x_d,\,x_s\}$ also satisfies \eqref{equ_dis_rela_mult}. In this case, $t_{peak}^{a,(l)}$ becomes smaller and smaller as $\{x_d,\,x_s\}$ gets closer and closer to $x_c^{(l)}$, where the distance between $x_d$ and $x_s$ is fixed. Furthermore, it reaches its minimum when 
\begin{equation}\label{xcs_center}
x_{c_1}^{(l)}=\frac{x_{d_1}+x_{s_1}}{2},\;\;x_{c_2}^{(l)}=\frac{x_{d_2}+x_{s_2}}{2}.
\end{equation}
As S-D pair $\{x_d,\,x_s\}$ moves, if there is another target $x_c^{(k)},\,1\leq k\neq l\leq J,$ such that 
$$|x_d-x_c^{(k)}|^2+|x_s-x_c^{(k)}|^2<|x_d-x_c^{(j)}|^2+|x_s-x_c^{(j)}|^2,\;\;1\leq j\neq k\leq J,$$
we can define its approximate peak time $t_{peak}^{a,(k)}$ by \eqref{tp_formulas}, which will have the same properties of $t_{peak}^{a,(l)}$ as described above when we move $\{x_d,\,x_s\}$. 

For the peak time $t_{peak}$, recall the numerical verification given in Subsection \ref{subsec_app_tps}.
From Figure \ref{fig_Apptps_2020} (d) and (e), we can observe that both $t_{peak}^a$ and $t_{peak}$ become smaller and smaller as $(m,\,n)$ approaches $(5,\,10)$ or $(15,\,10)$.
They reach local minima at $(m,\,n)=(5,\,10)$ and $(m,\,n)=(15,\,10)$, because the S-D pairs $\{x_d^{(5,10)},\,x_s^{(5,10)}\}$ and $\{x_d^{(15,10)},\,x_s^{(15,10)}\}$ are the nearest S-D pairs of \eqref{equ_gridpointSD} to targets $x_c^{(1)}$ and $x_c^{(2)}$, respectively. Furthermore, $\{x_d^{(5,10)},\,x_s^{(5,10)}\}$ and $\{x_d^{(15,10)},\,x_s^{(15,10)}\}$ satisfy  \eqref{xcs_center} for $x_c^{(1)}$ and $x_c^{(2)}$, respectively. From the relative error between $t_{peak}^a$ and $t_{peak}$ shown in Figure \ref{fig_Apptps_2020} (f), we reasonably speculate that $t_{peak}$ approximately satisfies the approximate peak time equation \eqref{tp_formulas} to any $x_c^{(j)},\,j=1,\,2$, where the chosen S-D pair satisfies \eqref{equ_dis_rela_mult}.

From the above discussion, the key point of reconstructing the first two coordinates of the targets is to search for the locally minimal peak time. Before proposing the reconstruction algorithm, we define the well-separated multiple point targets as follows.
\begin{definition}\label{def_wellsepa}
Let $x_c^{(j)}\in\Omega,\,j=1,\,2,\,\cdots,\,J,$ be the locations of unknown point targets, where $(x_{c_1}^{(j)},\,x_{c_2}^{(j)})\in {\rm ROI}:=(x_l,\,x_r)\times(x_b,\,x_t)$. 
%Introduce a uniform partition in ${\rm ROI}$ using a mesh
%\begin{equation*}
%x^{(m,n)}:=\left(m\cdot\frac{x_r-x_l}{M}+x_l,\, n\cdot\frac{x_t-x_b}{N}+x_b,\,0\right)
%\end{equation*}
%for $m=0,\,1,\,\cdots,\,M,\,n=0,\,1,\,\cdots,\,N.$
We define a set $\Xi$ of S-D pairs as
\begin{align}\label{SD_index}
\Xi:=\Big\{\{x_d^{(m,n)},\,x_s^{(m,n)}\}\,:\,&\{(x_l+\frac{L}{2}+m\frac{x_r-x_l}{M},\,x_b+n\frac{x_t-x_b}{N},\,0),\,
(x_l-\frac{L}{2}+m\frac{x_r-x_l}{M},\,x_b+n\frac{x_t-x_b}{N},\,0)\}\nonumber\\
&m=0,\,1,\,\cdots,\,M,\,n=0,\,1,\,\cdots,\,N,\,L>0\Big\}.
\end{align}
Let us further assume that there is a unique peak time for each $\{x_d^{(m,n)},\,x_s^{(m,n)}\}\in \Xi$. Then, there exists a set $\mathcal{P}$ of the peak times corresponding to $\Xi$.
We say that these $J$ unknown point targets are well-separated if there are $J$ local minimums in $\mathcal{P}$ for $J$ multiple point targets
\begin{equation}\label{tps}
t_{peak}^{min,(j)}:=t_{peak}^{min}(x_d^{min,(j)},\,x_s^{min,(j)};\,x_c^{(1)},\,\cdots,\,x_c^{(J)}),\;\;j=1,\,2,\,\cdots,\,J,
\end{equation}
where the S-D pair $\{x_d^{min,(j)},\,x_s^{min,(j)}\}\in \Xi$ is the S-D pair of the smallest distance from the $j$-th target in $\Xi$.
\end{definition}

We propose a boundary-scan reconstruction algorithm to reconstruct well-separated multiple point targets $x_c^{(j)},\,1\leq j\leq J,$ as in Algorithm \ref{algo_multpoints}. 

\begin{algorithm}[H]
\caption{Boundary-scan reconstruction algorithm for reconstructing multiple point targets}
\label{algo_multpoints}
\textbf{Input:} {\rm ROI}=$(x_l,\,x_r)\times(x_b,\,x_t)$,  distance $L$ of S-D pair.

\begin{description}
\item[Step 1] Scan ROI by using S-D pairs \eqref{SD_index}
%\begin{align}\label{SD_index}
%\Xi:=\Big\{\{x_d^{(m,n)},\,x_s^{(m,n)}\}\,:\,&(x_l+\frac{L}{2}+m\frac{x_r-x_l}{M},\,x_b+n\frac{x_t-x_b}{N},\,0),\,
%(x_l-\frac{L}{2}+m\frac{x_r-x_l}{M},\,x_b+n\frac{x_t-x_b}{N},\,0)\nonumber\\
%&m=0,\,1,\,\cdots,\,M,\,n=0,\,1,\,\cdots,\,N\Big\}
%\end{align}
and record the peak times $t_{peak}(x_d^{(m,n)},\,x_s^{(m,n)}),\,m=0,\,1,\,\cdots,\,M,\,n=0,\,1,\,\cdots,\,N.$

\item[Step 2] Reconstruct $(x_{c_1}^{(j)},\,x_{c_2}^{(j)}),\,j=1,\,2,\,\cdots,\,J,$ by 
\begin{equation}\label{c1c2_muti}
x_{c_1}^{(j),inv}:=\frac{x_{s_1}^{min,(j)}+x_{d_1}^{min,(j)}}{2},\qquad x_{c_2}^{(j),inv}:=\frac{x_{s_2}^{min,(j)}+x_{d_2}^{min,(j)}}{2},
\end{equation}
where $\{x_{d}^{min,(j)},\,x_{s}^{min,(j)}\}$ is a S-D pair of $\Xi$ at which peak time takes the local minimum $t_{peak}^{min,(j)}$. 

\item[Step 3] Reconstruct $x_{c_3}^{(j)},\,j=1,\,2,\,\cdots,\,J,$ by solving \eqref{tp_formulas} for given $\{x_d,\,x_s\}=\{x_{d}^{min,(j)},\,x_{s}^{min,(j)}\}$, $t=t_{peak}^{min,(j)}$ and $(x_{c_1}^{(j)},\,x_{c_2}^{(j)})=(x_{c_1}^{(j),inv},\,x_{c_2}^{(j),inv})$.

\end{description}
\end{algorithm}

\section{Numerical experiments}
\label{sec_exp}
In this section, we apply the proposed bisection reconstruction algorithm and boundary-scan reconstruction algorithm to the numerical experiments of single point and multiple point targets, respectively.
In all experiments, we set the physical parameters as \eqref{phys_para}.  
%$\mu_a=0.1\,{\rm mm^{-1}},\,v=0.219\,{\rm mm/ps},\,D=1/3\,{\rm mm},\,\beta=0.5493\,{\rm mm^{-1}}$ and $\ell=1000\,{\rm ps}$.
%$v,\,D,\,\beta,\,\ell$ as \eqref{phys_para} and $\mu_a=0.1\,{\rm mm^{-1}}$.
\medskip

%We apply the bisection reconstruction algorithm to reconstruct $x_c$. 
Usually, noise, specifically time jitters, is unavoidable in the measurements.
We perturb $t_{peak}$ by
\begin{equation}\label{noisydata_add}
t_{peak}^\delta:=\left(1+\hat\delta\times(2\times {\rm rand(1)} -1)\right)\times t_{peak},
\end{equation}
where $\hat\delta$ is the relative noise level, and $\rm rand(1)$ generates a uniformly distributed random number on the interval $(0,\,1)$.
We compute the relative error of the reconstruction by the following formula:
\begin{equation}\label{equ_relerr_xc}
RelErr=\sum_{j=1}^J\frac{|x_c^{(j)}-x_c^{(j),inv}|}{|x_c^{(j)}|},
\end{equation}
where $x_c^{(j)}$ and $x_c^{(j),inv}$ are the actual and reconstructed location of the $j$-th target, respectively.

%\subsection{Single point target}
Let us first consider the single point target case and suppress the superscript $(j)$ of $x_c^{(j)}$. 
We denote by ${\rm ROI}^{inv}$ the final ${\rm ROI}$ when the stop criteria of Algorithm \ref{algo_3D} and Algorithm \ref{algo_2D} are satisfied.

\begin{example}\label{exm01}
Let $x_c=(7,\,17,\,20)$. By setting {\rm ROI}=$(0,\,20)\times(0,\,20)$, $L=8$ and $\epsilon_1=\epsilon_2=0.1$ in Algorithm \ref{algo_3D}, we show the results of numerical reconstructions for different noise levels in Table \ref{tbl:ex01}. 
\end{example}

\begin{table}[htp]
\centering
\caption{Reconstructions for different noise  levels and fixed tolerances $\epsilon_1=\epsilon_2=0.1$}
\label{tbl:ex01}
\begin{tabular}{ccccc}
\hline
$x_c$& $\hat\delta$ & $x_c^{inv}$  &  $RelErr$ & ${\rm ROI}^{inv}$   \\ \hline
  \multirow{4}*{(7,\,17,\,20)} &0& (7.03,\,17.03,\,19.83) & 6.46e-03 & $(6.88,\,7.19)\times(16.88,\,17.19)$ \\
                           & 0.1\%&(6.76,\,17.30,\,19.79)  & 1.62e-02 &$(6.72,\,6.80)\times(17.27,\,17.34)$\\
                           &1\%&(7.54,\,16.99,\,19.68)   & 2.31e-02 & $(7.50,\,7.58)\times(16.95,\,17.03)$\\
                          &5\%&(7.85,\,17.30,\,19.08)   & 4.74e-02 & $(7.81,\,7.89)\times(17.27,\,17.34)$ \\\hline
\end{tabular}
\end{table}

For noise-free peak time ($\hat{\delta}=0$), the reconstruction is very close to the actual location, and its first two coordinates are also included in ${\rm ROI}^{inv}$, which shows that the proposed algorithm is accurate. Considering the case of $\hat{\delta}>0$, although the first two coordinates of the reconstructions are not included in ${\rm ROI}^{inv}$, they are still close to the true coordinates.
The relative errors of these reconstructions from the noisy peak time are still marginal, less than a few percent. 

We mention that the maximum number of shrinks is 8 times for $\epsilon_1=\epsilon_2=0.1$ in \textbf{Stage} 1. Next, we reduce the number of shrinks by setting larger tolerances $\epsilon_1=\epsilon_2=1.25$ in Example \ref{exm01}. We can still have reasonable numerical results, which are shown in Table \ref{tbl:ex012}.

\begin{table}[htp]
\centering
\caption{Reconstructions for different noise  levels and fixed tolerances $\epsilon_1=\epsilon_2=1.25$}
\label{tbl:ex012}
\begin{tabular}{ccccc}
\hline
$x_c$& $\hat\delta$ & $x_c^{inv}$  &  $RelErr$ & ${\rm ROI}^{inv}$   \\ \hline
  \multirow{4}*{(7,\,17,\,20)} &0& (6.88,\,16.88,\,19.81) & 9.51e-03 & $(6.25,\,7.50)\times(16.25,\,17.50)$ \\
                           & 0.1\%&(6.88,\,16.88,\,19.85)  & 8.06e-03 &$(6.25,\,7.50)\times(16.25,\,17.50)$\\
                           &1\%&(8.13,\,16.88,\,20.07)   & 4.38e-02 & $(7.50,\,8.75)\times(16.25,\,17.50)$\\
                          &5\%&(8.13,\,16.88,\,20.56)   & 4.65e-02 & $(7.50,\,8.75)\times(16.25,\,17.50)$ \\\hline
\end{tabular}
\end{table}

For $\epsilon_1=\epsilon_2=1.25$, the number of shrinks reduces to 4 times in \textbf{Stage} 1. Compared with the results shown in Table \ref{tbl:ex01}, we obtain that a finite number of shrinks can also give a good result. Now, ${\rm ROI}^{inv}$ contains the true  location of target for $\hat{\delta}=0.1\%$. In this case, the relative error is smaller than the one shown in Table \ref{tbl:ex01}. 
However, the true location of the target is not included in ${\rm ROI}^{inv}$ for $\hat{\delta}=1\%$ and $\hat{\delta}=5\%$. It is easy to reason that there is a false shrink occurred at the third shrink. Even in this case, we can get the reconstructions similar to Table \ref{tbl:ex01}. In short,  larger tolerances for larger noise levels may result in a much better reconstruction of the first two coordinates $(x_{c_1}^{inv},\,x_{c_2}^{inv})$. 
Moreover, the reconstruction of the depth $x_{c_3}^{inv}$ is less affected by the reconstruction $(x_{c_1}^{inv},\,x_{c_2}^{inv})$. 
\medskip

The next experiment is to verify the proposed algorithm for a deeply embedded target.
\begin{example}
    Let $x_c=(6,\,11,\,30)$. %By setting {\rm ROI}=$(0,\,20)\times(0,\,20)$, $L=8$ and $\epsilon_1=\epsilon_2=0.1$ in Algorithm \ref{algo_3D}, 
    We show the results of numerical reconstructions for different noise levels in Table \ref{tbl:ex03} with the same inputs as Example \ref{exm01}.
\end{example}

\begin{table}[htp]
\centering
\caption{Reconstructions for different noise  levels and fixed tolerances $\epsilon_1=\epsilon_2=0.1$}
\label{tbl:ex03}
\begin{tabular}{ccccc}
\hline
$x_c$& $\hat\delta$ & $x_c^{inv}$  &  $RelErr$ & ${\rm ROI}^{inv}$   \\ \hline
 \multirow{4}*{(6,\,11,\,30)} &0& (6.09,\,10.78,\,30.17) & 9.06e-03 & $(5.94,\,6.25)\times(10.63,\,10.94)$ \\
                           & 0.1\%&(5.66,\,11.37,\,30.17)  & 1.62e-02 &$(5.63,\,5.70)\times(11.33,\,11.41)$\\
                           &1\%&(7.54,\,11.99,\,30.02)   & 5.63e-02 & $(7.50,\,7.58)\times(11.95,\,12.03)$\\
                          &5\%&(7.85,\,17.30,\,29.79)   & 2.02e-01 & $(7.81,\,7.89)\times(17.27,\,17.34)$ \\\hline
\end{tabular}
\end{table}

For $\hat{\delta}=0$, we obtain a similar result as Example \ref{exm01}, that the reconstruction is very close to the true location. %, and its first two coordinates are also included in ${\rm ROI}^{inv}$. 
For $\hat{\delta}=0.1\%$ and $\hat{\delta}=1\%$, the reconstructions approximate the actual location within small relative errors. 
The reconstruction becomes worse for $\hat{\delta}=5\%$, since the peak time is more delayed as the target depth increases, causing an increase in the order of noise. Comparing the results shown in Example \ref{exm01}, even for a deeper unknown target, the reconstruction $x_{c_3}^{inv}$ is less affected by the noise. 

%\subsection{Multiple point targets}

Next, we use the following numerical example to verify the performance of Algorithm \ref{algo_multpoints} for reconstructing the locations of two unknown point targets. 

\begin{example}\label{exa_mult01}
    Let $x_c^{(1)}=(3.3,\,5.2,\,16)$, $x_c^{(2)}=(17.4,\,16.7,\,18)$ and ${\rm ROI}:=(0,\,20)\times(0,\,20)$. We choose the S-D pairs as \eqref{SD_index} for $M=N=20,\,L=2$. The numerical results are shown in Table \ref{tbl:ex_mult} for different noise levels.
\end{example}

\begin{table}[htp]
\centering
\caption{Reconstructions for different noise levels in Example \ref{exa_mult01}}
\label{tbl:ex_mult}
\begin{tabular}{ccccc}
\hline
$x_c^{(1)},x_c^{(2)}$& $\hat\delta$ & $x_c^{(1),inv}$ & $x_c^{(2),inv}$ &  $RelErr$  \\ \hline
  \multirow{3}*{(3.3,\,5.2,\,16),\,(17.4,\,16.7,\,18)} &0&(3.00,\,5.00,\,15.67) & (17.00,\,17.00,\,17.72) &  4.75e-02\\
                           & 0.1\%& (3.00,\,5.00,\,15.71)  & (17.00,\,17.00,\,17.79) &4.51e-02\\
                           &1\%&  (3.00,\,6.00,\,15.70)    & (18.00,\,17.00,\,17.84)& 7.58e-02\\\hline

\end{tabular}
\end{table}

From Figure \ref{fig_mutl}, there exist two local minima $t_{peak}^{min,(1)}=546.1$ and $t_{peak}^{min,(2)}=603.5$ measured by S-D pairs $\{x_d^{(3,5)},\,x_s^{(3,5)}\}$ and $\{x_d^{(17,17)},\,x_s^{(17,17)}\}$, respectively, where the scanning S-D pairs are defined by \eqref{SD_index}. Hence, we can reconstruct these two point targets by using Algorithm \ref{algo_multpoints}.
For the measured peak times containing noise, we first apply the moving average method with a $3\times 3$ grid to smooth the noisy peak times since the algorithm will fail due to the presence of several local minimal peak times. For instance, looking at the cross-section of the peak time, the noisy peak time, and the smoothed noisy peak time shown in Figure \ref{fig_mutl_noise}, several local minimal peak times in $t_{peak}^{\delta}(x_d^{(m,5)},\,x_s^{(m,5)})$ and $t_{peak}^{\delta}(x_d^{(m,17)},\,x_s^{(m,17)})$ for $m=0,\,1,\,\cdots,\,20$ are clearly visible. After smoothing the noisy peak times, we can distinguish two local minima in the noisy peak times $t_{peak}^\delta(x_d,\,x_s),\,\{x_d,\,x_s\}\in \Xi$. Then, we can reconstruct the locations of targets with a similar discussion for the noise-free peak times. 
The numerical results verify that the proposed algorithm is feasible to reconstruct the locations of multiple point targets.  

\begin{figure}[htp]
\centering
\begin{tabular}{ll}
(a)&(b)  \\
\includegraphics[width=0.5\textwidth]{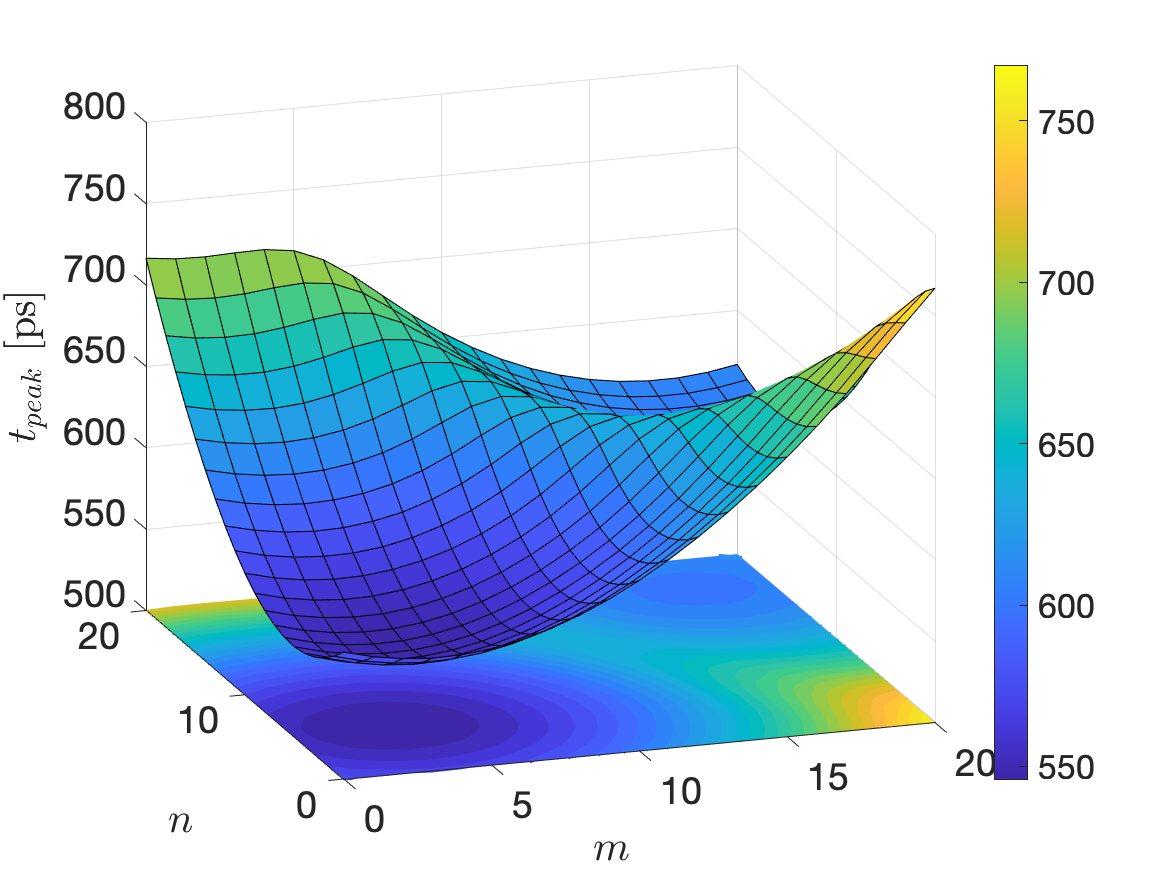}
&\includegraphics[width=0.5\textwidth]{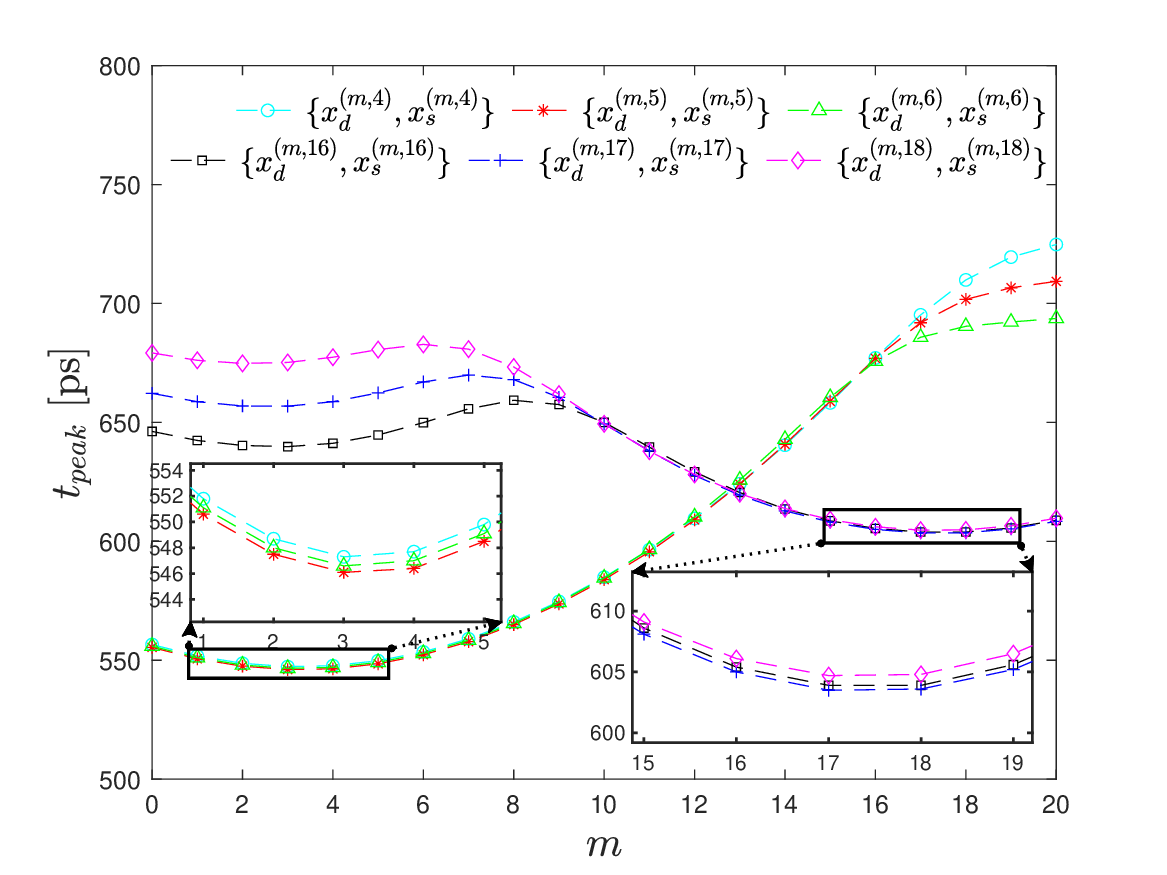}
\end{tabular}
\caption{Noise-free peak times measured by different S-D pairs \eqref{SD_index} in Example \ref{exa_mult01}: (a) 3D contour plot of peak times for $m,\,n=0,\,1,\,\cdots,\,20$, (b) peak times for $m=0,\,1,\,\cdots,\,20,\,n=4,\,5,\,6,\,16,\,17,\,18$}
\label{fig_mutl}
\end{figure}

\begin{figure}[htp]
\centering
\begin{tabular}{ll}
(a)&(b)  \\
\includegraphics[width=0.5\textwidth]{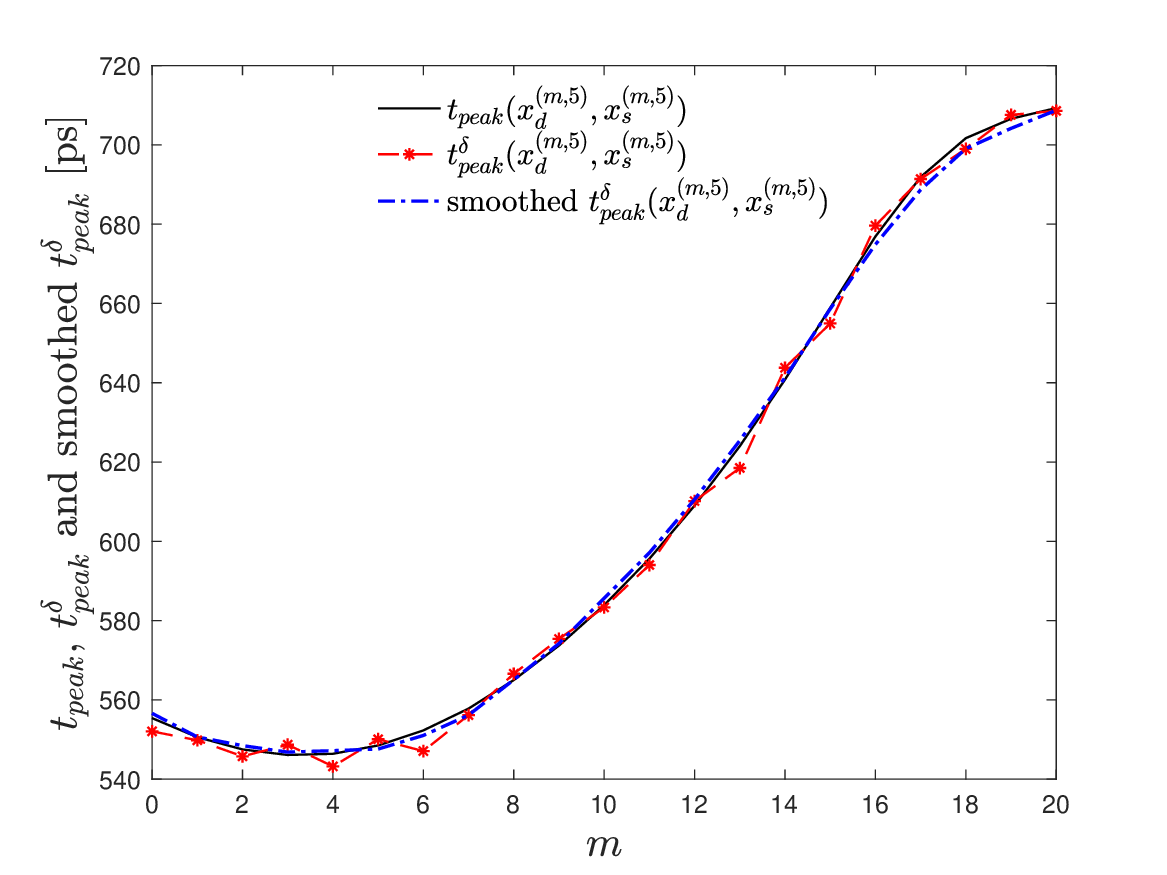}
&\includegraphics[width=0.5\textwidth]{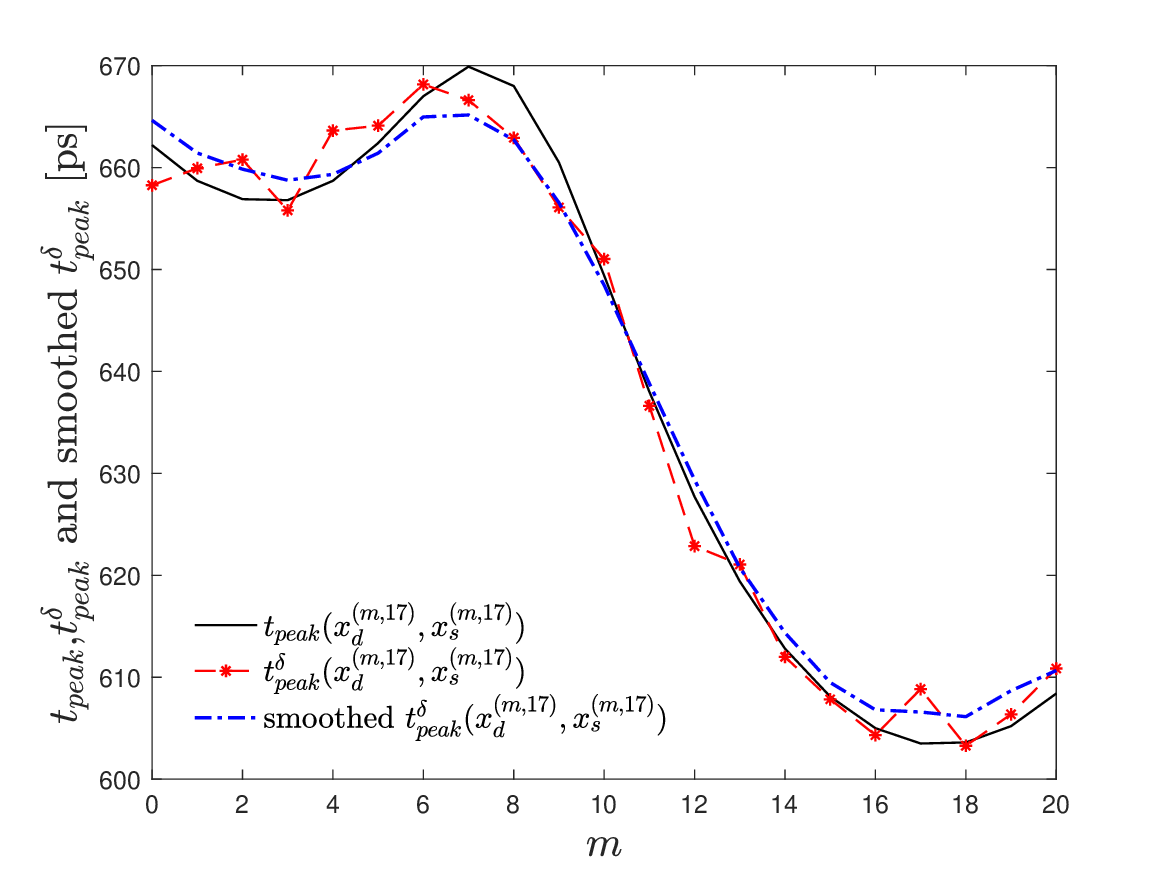}\\
\end{tabular}
\caption{Peak times, noisy peak times ($\hat{\delta}=0.1\%$) and smoothed noisy peak times for $m=0,\,1,\,2,\,\cdots,\,20 $ and $n=5,\,17$}
\label{fig_mutl_noise}
\end{figure}

\section{Conclusions and future work}\label{sec_conc}
In this paper, we studied the reconstruction of the locations of unknown point targets from peak times, which is the time when the temporal response function $U_m(x_d,\,t;\,x_s)$ reaches its maximum for a given S-D pair $\{x_d,\,x_s\}$ on the measurement surface. Further, we considered the peak time for the nonzero lifetime $\ell>0$ as an extension of our previous paper \cite{Chen2023}. 
Analyzing the asymptotic behavior of $U_m$, we derived the approximate peak time equation as a nonlinear equation and could define the approximate peak time by its positive root.   We proved the properties of approximate peak time (uniqueness, order relation, symmetry, and accuracy) and provided numerical verification of them. Based on these analyses, we proposed bisection reconstruction and boundary-scan algorithms to reconstruct the location of the point targets. %where the peak time is used as the index to shrink ${\rm ROI}$. %Under the well-separated assumption between targets, we extended the bisection reconstruction algorithm to multiple point targets. 
We remark on the advantage of considering the peak time. The peak time is uniquely and visibly identified. It is the least noisy data point in $U_m(x_d,\,t;\,x_s)$, making it potentially robust for reconstruction. 

The novelties of this paper can be summarized as follows
\begin{itemize}
\item The case of $\ell>0$ makes the mathematical model \eqref{ue_sys}--\eqref{um_sys} better fit the physical processes of FDOT. 
\item We found no other theoretical study considering the effect of $\ell$ on the peak time and defining the approximate peak time. The obtained analytical result is reliable and convenient to be applied to our inverse problem.
\item The proposed bisection reconstruction and boundary-scan algorithms are less time-consuming, efficient, robust, and accurate for identifying the locations of the point targets without any regularization. 
%\item The derived asymptotic analysis is easily extended to the case of the multiple point targets. {\color{magenta}This is insufficient!!}
\end{itemize}

We can further say that based on our asymptotic analysis, we have established reconstructing the locations of multiple points targets using the relation between the peak time and the S-D pairs. In the analysis, the distance function $d(x_c,t):=(|x_d-x_c|^2+|x_s-x_c|^2)/t$ plays the most important role, where $\{x_d,\,x_s\}$ is the S-D pair. Actually, the nearest point target $x_c$ to $\{x_d,\,x_s\}$ is on the sphere centered at $(x_d+x_s)/2$ with the radius depending on $t$ which is the level surface of the distance function, and we have used this to define the well-separated multiple point targets. 

As for much more practical situations, we will consider the following cases in our next FDOT study. They are the cases where the targets are not point targets, and the measurement surface $\partial\Omega$ is curved. The first task to start this study will be to have the Green function for the FDOT under the setup, including these cases. This is available by easily modifying the argument in one of our coauthors' papers \cite{Nakamura} giving the Green function for the interior transmission problem. The advantage of the mentioned argument is based on using the parabolic scaling which immediately gives the dominant part of the Green function. In relation to this, we note that the distance function is invariant under the parabolic scaling. Assuming the targets are well-separated convex domains and looking at the dominant part of the Green function, we speculate that we will find a similar situation as for the point target case in a neighborhood of the point that the mentioned sphere touches the target.

%\section*{Data availability statement}
%Data underlying the results presented in this paper are not publicly available at this time but may be obtained from the authors upon reasonable request.

%\section*{ORCID iDs}
%\begin{description}
%\item  Shuli Chen: https://orcid.org/0009-0000-6563-5312 
%\item Junyong Eom: https://orcid.org/0000-0002-2749-3322
%\item Gen Nakamura:  https://orcid.org/0000-0002-7911-8612
%\item Goro Nishimura: https://orcid.org/0000-0003-4330-2626 
%\end{description}

\section*{Author contributions}
Shuli Chen: Formal analysis, Numerical simulation, Writing - original draft; Junyong Eom: Conceptualization, Mathematical analysis, Validation, Writing - original draft; Nakamura Gen: Proposing asymptotic methods for Section 2, Writing, Editing; Nishimura Goro: Proposing the idea of this study, Physical interpretation of the results, Writing - review and editing. 
All authors have
read and agreed to the published version of the manuscript

\end{document}